\documentclass{tac}
\usepackage{amsmath}
\usepackage{amssymb}
\usepackage{latexsym}
\usepackage[all,cmtip]{xy}
\usepackage{graphicx, float, color}

\usepackage{xcolor}

\begin{document}
\title{On Deformations of Pasting Diagrams, II}

\copyrightyear{2013}

\thanks{}

\eaddress{dyetter@math.ksu.edu}

\author{Tej Shrestha \\ D.\ N.\ Yetter}
\cauthor{D.\ N.\ Yetter}

\address{Department of Mathematics\\
Ohio University\\ Athens, OH 45701 \\ \\
Department of Mathematics \\
Kansas State University \\ Manhattan, KS 66506}

\bibliographystyle{plain}

\maketitle

\begin{abstract} We continue the development of the infinitesimal deformation theory of pasting diagrams of $k$-linear categories begun in \cite{Yetter}.  In \cite{Yetter} the standard result that all obstructions are cocycles was established only for the elementary, composition-free parts of pasting diagrams.  In the present work we give a proof for pasting diagrams in general.  As tools we use the method developed by Shrestha \cite{Shrestha} of representing formulas for obstructions, along with the corresponding cocycle and cobounding conditions by suitably labeled polygons, giving a rigorous exposition of the previously heuristic method, and deformations of pasting diagrams in which some cells are required to be deformed trivially. 
\end{abstract}

\section{Introduction}

It is the purpose of this paper to extend the results of Yetter \cite{Yetter}, generalizing classical results of Gerstenhaber \cite{G} and Gerstenhaber and Schack \cite{GS} on the infinitesimal deformation theory of associative algebras and poset-indexed diagrams of associative algebras to a deformation theory for arbitrary pasting diagrams of $k$-linear categories, $k$-linear functors, and natural transformations.  In particular, in \cite{Yetter} the standard result that obstructions are cocycles was established only for the simplest parts of pasting diagrams:  for pasting diagrams in which no compositions either 1- or 2-dimensional occur.  In this paper we will establish it for deformation complexes of pasting diagrams in general, by first giving a detailed and rigorous exposition of a method developed heuristically by Shrestha \cite{Shrestha}, then applying the method to prove that obstructions are all cocycles in a family of pasting diagrams sufficient to imply the result in general.

Along the way to proving that obstructions are closed in general, we will have occasion to consider deformations of pasting diagrams in which specified functors or natural transformations are required to be deformed trivially.  Although in the present work such conditions will be used only to reduce the problem of showing obstructions are cocycles to simple instances, the ability to handle deformations subject to such restrictions could well be useful in other settings.

We will also make explicit a point overlooked in the statements and proofs of \cite{Yetter} Theorems 8.2 and 8.3:  the cochain maps constructed in those theorems depend on choices of association for 2-dimenensional compositions in the pasting digaram.  However, as we establish here, the cochain maps are independent of those choices, up to algebraic homotopy, and thus, the isomorphism type of the deformation complex for a pasting diagram is well-defined in either the homotopy category or derived category of cochain complexes over $k$.

\section{Chain maps from pasting composition}

At the level of the Hochschild cochain complexes $C^\bullet(F,G)$, the interesting cochain maps are described fully in \cite{Yetter}.  Our purpose in this section is, rather, to describe chain maps induced by pasting compositions on the full deformation complexes ${\frak C}^\bullet (D)$, when $D$ is a composable pasting diagram.

Proposition 4.5 of \cite{Yetter} constructs two chain maps:\smallskip

If $F,G:{\cal C}\rightarrow {\cal D}$, and $H:{\cal D}\rightarrow
{\cal E}$
are functors, then there is a cochain map 
$H_*(-):C^\bullet(F,G)\rightarrow
C^\bullet(H(F),H(G))$
given by 

\[ H_*(\phi)(f_1,\ldots ,f_n) := H(\phi(f_1,\ldots ,f_n)). \]

Similarly if $F:{\cal C}\rightarrow {\cal D}$ and 
$G,H:{\cal D}\rightarrow {\cal E}$ are functors, there is a cochain
map
$F^*(-) = -(F^\bullet):C^\bullet(G,H)\rightarrow
C^\bullet(G(F),H(F))$ given by

\[ F^*(\phi)(f_1,\ldots , f_n) = \phi(F^\bullet)(f_1,\ldots , f_n) := 
\phi(F(f_1),\ldots ,F(f_n)). \]\smallskip

 Proposition 4.6 of \cite{Yetter} gives two more:\smallskip

If $\tau:F_1\Rightarrow F_2$ is a natural transformation, then
post- (resp. pre-) composition by $\tau$ induces a cochain map
$\tau^* :C^\bullet(F_2,G)\rightarrow C^\bullet(F_1,G)$ 
(resp. $\tau_*:C^\bullet(G,F_1)\rightarrow C^\bullet(G,F_2)$
for any
functor $G$. \smallskip

Here 

\[ \tau^{*}(\phi) = \tau \cup \phi \]

\noindent and

\[ \tau_{*}(\phi) = \phi \cup \tau. \]

And, \cite{Yetter} Proposition 4.7 give a cochain map, which ties together all of the deformations when a natural transformation, its source and target, and their common source and target are deformed simultaneously:

Let $\sigma:F\Rightarrow G$ be a natural transformation between $k$-linear functors $F,G:{\cal A} \rightarrow {\cal B}$.

Let $C^\bullet({\cal A} \stackrel{F}{G} {\cal B})$ denote the cone on the cochain map

\[ \left[ \begin{array}{cc}
                  F_* & -F^* \\               
                  G_* & -G^*  
               \end{array} \right]: 
             C^\bullet({\cal A})\oplus C^\bullet({\cal B})
               \rightarrow C^\bullet(F) \oplus C^\bullet(G). \]

The cochain groups are

\[ C^\bullet({\cal A} \stackrel{F}{G} {\cal B}) :=
C^{\bullet+1}({\cal A})
\oplus C^{\bullet+1}({\cal B})\oplus C^\bullet(F)\oplus C^\bullet(G)
\]

with coboundary operators given by

\[ d_{{\cal A} \stackrel{F}{G} {\cal B}} = \left[
    \begin{array}{cccc}
	-d_{\cal A} & 0 & 0 & 0 \\
	0 & -d_{\cal B} & 0 & 0\\
	-F_{*} & F^{*} & d_{F} & 0\\
	-G_{*} & G^{*} & 0 & d_{G}
     \end{array} \right] \]

Proposition 4.7\footnote{The negative sign on $\sigma^*$ was omitted in the statement of the Proposition, though it is plainly present in the proof given in \cite{Yetter}} was then 
\smallskip

Let $\sigma:F\Rightarrow G$ be a natural
transformation, 
then

\[ \sigma \ddagger := 
\left[ \begin{array}{cccc} 0 & (-)\{\sigma \} & \sigma_*
&  -\sigma^* \end{array} \right]: 
C^\bullet({\cal A} \stackrel{F}{G} {\cal B}) \rightarrow
C^\bullet(F,G) \]

\noindent is a cochain map.	
\smallskip

In \cite{Yetter} the chain maps of Propositions 8.1, 8.2 and 8.3 were used only to construct the deformation complex of a general pasting diagram.  In fact, they can be assembled into chain maps from the deformation complex of a composable pasting diagram to the simpler pasting diagram in which the compositions have been carried out.  Propositions \ref{chainmapfrom2comp}, \ref{chainmapfrom1precomp} and \ref{chainmapfrom1postcomp} give the map explicitly in the cases of a single 2 composition, precomposition of a natural transformation by a functor, and postcompostion of a natural transformation by a functor, respectively.

In each case the proof begins by collecting the maps from the Propositions of \cite{Yetter} with values in the direct summands corresponding to cells of the pasting diagram in which the compositions have been performed, together with identity maps for those cells which remain from the original diagram, and arranging their summands in the correct places of a matrix
of maps.  What is indicated in the sketches of proofs following each proposition are the main difficulties in the unedifying calculation which shows that the result is, in fact, a chain map between the deformation complexes.


\begin{proposition}  \label{chainmapfrom2comp}

Let $F,G,H:{\cal A}\rightarrow {\cal B}$ be $k$-linear functors, and $\sigma:F\Rightarrow G$ and $\tau:G\Rightarrow H$ be natural transformations.  Let $D$ be the pasting diagram consisting of both $\sigma$ and $\tau$ and their (iterated) sources and targets, and let $D^\prime$ be the pasting diagram consisting of the 2-dimensional composition $\sigma\tau$ and its (iterated) sources and targets.  Then there is a chain map

 \[\circ_2^\bullet: {\frak C}^\bullet(D)\rightarrow {\frak C}^\bullet(D^\prime) \] 

\noindent induced by the 2-dimensional composition.  

In particular if the summands of $ {\frak C}^\bullet(D)$ (resp. $ {\frak C}^\bullet(D^\prime)$) are given in the order

\[ C^{\bullet+2}({\cal A}) \oplus C^{\bullet+2}({\cal B})\oplus C^{\bullet+1}(F) \oplus 
C^{\bullet+1}(G) \oplus C^{\bullet+1}(H) \oplus C^{\bullet}(F,G) \oplus C^{\bullet}(G,H)\]

\noindent (resp.

\[C^{\bullet+2}({\cal A}) \oplus C^{\bullet+2}({\cal B})\oplus C^{\bullet+1}(F)\oplus 
C^{\bullet+1}(H)\oplus C^{\bullet}(F,H) \; \mbox{\it )}\]

\noindent $\circ_2^\bullet$ is given by

\[ \left[
    \begin{array}{ccccccc}
Id_{C^{\bullet+2}({\cal A})} & 0 & 0 & 0 & 0 & 0 & 0 \\
0 & Id_{C^{\bullet+2}({\cal B})} & 0 & 0 & 0 & 0 & 0 \\
0 & 0 & Id_{C^{\bullet+1}(F)} & 0 & 0 & 0 & 0 \\
0 & 0 & 0 & 0 & Id_{C^{\bullet+1}(H)} & 0 & 0 \\
0 & (-)\{\sigma,\tau\} & 0 & 0 & 0 & \tau_* & \sigma^*
     \end{array} \right]  \]

\end{proposition}

\begin{proof} (Sketch)
The only subtlety in the completely computational verification that this is a chain map involves one coordinate in which a relation of the sort in Gerstenhaber and Voronov \cite{GV} or the proof of Proposition 4.7 in \cite{Yetter} is needed. \end{proof}

\begin{proposition} \label{chainmapfrom1precomp}

Let $F:{\cal A}\rightarrow {\cal B}$ and $G,H:{\cal B}\rightarrow {\cal C}$ be $k$-linear functors, and $\sigma:G\Rightarrow H$ be a natural transformation.  Let $D$ be the pasting diagram consisting of $\sigma$ and its (iterated) sources and targets together with $F$ and $\cal A$, and let $D^\prime$ be the pasting diagram consisting of the 1-dimensional composition $\sigma_F$ and its (iterated) sources and targets.  Then there is a chain map

 \[\circ_{1,l}^\bullet: {\frak C}^\bullet(D)\rightarrow {\frak C}^\bullet(D^\prime) \] 

\noindent induced by the 1-dimensional composition.  

In particular if the summands of $ {\frak C}^\bullet(D)$ (resp. $ {\frak C}^\bullet(D^\prime)$) are given in the order

\[ C^{\bullet+2}({\cal A}) \oplus C^{\bullet+2}({\cal B}) \oplus C^{\bullet+2}({\cal C})\oplus C^{\bullet+1}(F) \oplus 
C^{\bullet+1}(G) \oplus C^{\bullet+1}(H) \oplus C^{\bullet}(G,H)\]

\noindent (resp.

\[C^{\bullet+2}({\cal A}) \oplus C^{\bullet+2}({\cal C})\oplus C^{\bullet+1}(G(F))\oplus 
C^{\bullet+1}(H(F))\oplus C^{\bullet}(G(F),H(F)) \; \mbox{\it )}\]

\noindent $\circ_{1,l}^\bullet$ is given by

\[ \left[
    \begin{array}{ccccccc}
Id_{C^{\bullet+2}({\cal A})} & 0 & 0 & 0 & 0 & 0 & 0 \\
0 & 0 & Id_{C^{\bullet+2}({\cal C})} & 0 & 0 & 0 & 0 \\
0 & 0 & 0 & G_* & F^*  & 0 & 0 \\
0 & 0 & 0 & H_* & 0 & F^* & 0 \\
0 & 0 & 0 & 0 & 0 & 0 & F^*
     \end{array} \right]  \]

\end{proposition}

\begin{proof}(Sketch)
Here the ``hardest'' verification is a coordinate which vanishes by the naturality of $\sigma_F$.
\end{proof}

\begin{proposition}  \label{chainmapfrom1postcomp}

Let $F,G:{\cal A}\rightarrow {\cal B}$ and $H:{\cal B}\rightarrow {\cal C}$ be $k$-linear functors, and $\sigma:F\Rightarrow G$ be a natural transformation.  Let $D$ be the pasting diagram consisting of $\sigma$ and its (iterated) sources and targets, together with $H$ and $\cal C$, and let $D^\prime$ be the pasting diagram consisting of the 1-dimensional composition $\tau = H(\sigma)$ and its (iterated) sources and targets.  Then there is a chain map

 \[\circ_{1,r}^\bullet: {\frak C}^\bullet(D)\rightarrow {\frak C}^\bullet(D^\prime) \] 

\noindent induced by the 2-dimensional composition.  

In particular if the summands of $ {\frak C}^\bullet(D)$ (resp. $ {\frak C}^\bullet(D^\prime)$) are given in the order

\[ C^{\bullet+2}({\cal A}) \oplus C^{\bullet+2}({\cal B}) \oplus C^{\bullet+2}({\cal C})\oplus C^{\bullet+1}(F) \oplus 
C^{\bullet+1}(G) \oplus C^{\bullet+1}(H) \oplus C^{\bullet}(F,G)\]

\noindent (resp.

\[C^{\bullet+2}({\cal A}) \oplus C^{\bullet+2}({\cal C})\oplus C^{\bullet+1}(H(F))\oplus 
C^{\bullet+1}(H(G))\oplus C^{\bullet}(H(F),H(G)) \; \mbox{\it )}\]

\noindent $\circ_{1,r}^\bullet$ is given by

\[ \left[
    \begin{array}{ccccccc}
Id_{C^{\bullet+2}({\cal A})} & 0 & 0 & 0 & 0 & 0 & 0 \\
0  & 0 & Id_{C^{\bullet+2}({\cal C})} & 0 & 0 & 0 & 0 \\
0 & 0 & 0 & H_* & 0 & F^* & 0 \\
0 & 0 & 0 & 0 & H_* & G^* & 0 \\
0 & 0 & 0 & 0 & 0 & (-)\{\sigma\} & H_*
     \end{array} \right] \]

\end{proposition}

\begin{proof} (Sketch)
Here, almost all of the coordinates in the bottom row present some minor ``difficulty'', one depends on a Gerstenhaber-Voronov-style relation, while most of the others require unpacking the definitions of pullback or pushforward maps along functors and natural transformations to see some commutativity relation between them.
\end{proof}

\section{Examples and a note on 2-associativity}

In this section we use the cochain maps of the previous section to give explicit examples of deformation complexes associated to various shapes commutative pasting diagrams.  In one example we see
that cochain maps induced by two different associations of the 2-dimensional composition
are not actually equal, but only chain homotopic (or, to put it another way, are not equal
in the abelian category of cochain complexes, but {\em are} in the homotopy category of cochain complexes and thus {\em a fortiori} in the derived category).

In the examples, we use commutative pasting diagrams because only in that context is the effect of the compositions of 2-arrows evident:  if no commutativities are enforced, the iterative cone construction given explicitly in \cite{Yetter} Proposition 4.7 suffices to describe the entire structure of the deformation complex of the pasting diagram, as indeed it does in the case of
commutative pasting diagrams in which no compositions occur (for example, the commutative ``pillow'' consisting of two categories a pair of parallel functors between them, and two copies of the same natural transformation between the functors).  Only when compositions are involved do the cochain maps of the previous section play a role.

In what follows, we leave zero entries in matrices of maps giving coboundaries blank, as this seems to improve readability.


We begin with the simplest example, a ``commutative pillow with triangular cross-section'':

\begin{example} \label{triangularpillow}
Consider the pasting diagram given by three functors $F,G,H: {\cal A}\rightarrow {\cal B}$
and three natural transformations $\sigma:F \Rightarrow G$, $\tau:G \Rightarrow H$ and
$\upsilon:F \Rightarrow H$ , with the obvious 0-, 1- and
2-cells and a single 3-cell enforcing the condition $\sigma \tau = \upsilon$.

The deformation complex of the pasting diagram is then given by

\[ C^{\bullet+3}({\cal A}) \oplus C^{\bullet+3}({\cal B})\oplus C^{\bullet+2}(F) \oplus 
C^{\bullet+2}(G) \oplus C^{\bullet+2}(H) \oplus C^{\bullet+1}(F,G) \oplus C^{\bullet+1}(G,H)\]
\[\oplus C^{\bullet+1}(F,H) \oplus C^{\bullet}(F,H)\]

\noindent
with coboundary given by

\[ \left[
    \begin{array}{ccccccccc}
-d_{\cal A} & & & & & & & & \\
& -d_{\cal B} & & & & & & &\\
-F_* & F^* & d_F & & & & & &\\
-G_* & G^* & & d_G & & & & &\\
-H_* & H^* & & & d_H & & & &\\
& (\cdot)\{\sigma\} & \sigma_* & -\sigma^* & & -d_{F,G} & & & \\
& (\cdot)\{\tau\} &  &\tau_* & -\tau^* & & -d_{G,H} & & \\
& (\cdot)\{\upsilon\} & \upsilon_* & & -\upsilon^* & & & -d_{F,H} & \\
& (\cdot)\{\sigma, \tau\} & & & & \tau_* & \sigma^* & -Id & d_{F,H}
 \end{array} \right] \]

\end{example}

Note that here we have chosen to consider the composition $\sigma \tau$ to be the source of the 3-cell and the single natural transformation $\upsilon$ to be the target.  For the opposite choice the complex would be the same, except that the non-diagonal entries of the last row would be negated.

The complex ${\frak C}^\bullet(P)$ of Example \ref{triangularpillow} is related to the map $\Phi$ of Proposition \ref{chainmapfrom2comp} by

\begin{proposition} \label{babyqi}
${\frak C}^\bullet(P)$ is quasi-isomorphic to the (dual) mapping cylinder on $\Phi$.
\end{proposition}

\begin{proof} This follows immediately by applying the following general lemma about (dual) mapping cylinders of chain maps between mapping cones.
\end{proof}

\begin{lemma}
Let $f:C^\bullet \rightarrow A^\bullet$ and $g:C^\bullet \rightarrow B^\bullet$ be chain maps, and suppose $\phi: cone(f) \rightarrow cone(g)$ is a chain map between their 
mapping cones of the form

\[ \phi = \left[ \begin{array}{cc} 
                              Id_{C^\bullet} & \\
				\sigma & \tau \end{array} \right] .\]

Then the (dual) mapping cylinder on $\phi$, $cone(f)[1] \oplus cone(g)[1] \oplus cone(g)$ with 
differential given by

\[ \left[ \begin{array}{ccc}
		-d_{cone(f)} &  &  \\
		 & -d_{cone(g)} &  \\
		-\phi & Id_{cone(g)} & d_{cone(g)} \end{array} \right] ,\]

is quasi-isomorphic to $C^{\bullet+2} \oplus A^{\bullet +1} \oplus B^{\bullet +1}  \oplus B^{\bullet} $ with differential given by

\[ \left[ \begin{array}{cccc}
 		-d_C & & & \\
		f & d_A & & \\
		g & & d_B & \\
		-\sigma & -\tau & Id_B & -d_B \end{array} \right] .\]
\end{lemma}

\begin{proof}
Include the last complex into the (dual) mapping cylinder by $(c, a, b, b^\prime)^T \mapsto
(c, a, c, b, 0, b^\prime)^T$, the cokernel is plainly isomorphic to $cone(Id_{C[1]})$, which is acylic, so the subcomplex is quasi-isomorphic to the (dual) mapping cylinder as claimed.
\end{proof}

Continuing with pasting diagrams which only involve 2-dimensional compositions, consider next the two ``pillows with square cross-sections'':

\begin{example}
Consider the pasting diagram given by four functors $F,G,G^\prime, H: {\cal A}\rightarrow {\cal B}$
and four natural transformations $\sigma:F \Rightarrow G$, $\tau:G \Rightarrow H$,
$\sigma^\prime:F \Rightarrow G^\prime$ and $\tau^\prime:G^\prime \Rightarrow H$, with the obvious 0-, 1- and
2-cells and a single 3-cell enforcing the condition $\sigma \tau = \sigma^\prime \tau^\prime$.

The deformation complex of the pasting diagram is then given by

\[ C^{\bullet+3}({\cal A}) \oplus C^{\bullet+3}({\cal B})\oplus C^{\bullet+2}(F) \oplus 
C^{\bullet+2}(G) \oplus C^{\bullet+2}(G^\prime) \oplus C^{\bullet+2}(H) \oplus C^{\bullet+1}(F,G) \] 
\[\oplus C^{\bullet+1}(G,H)
\oplus C^{\bullet+1}(F,G^\prime) \oplus C^{\bullet+1}(G^\prime,H) \oplus C^{\bullet}(F,H)\]

\noindent
with coboundary given by

\[ \left[
    \begin{array}{ccccccccccc}
-d_{\cal A} & & & & & & & & & &\\
& -d_{\cal B} & & & & & & & & & \\
-F_* & F^* & d_F & & & & & & & &\\
-G_* & G^* & & d_G & & & & & & &\\
-G^\prime_* & G^{\prime *} & & & d_{G^\prime} & & & & & &\\
-H_* & H^* & & & & d_H & & & & & \\
& (\cdot)\{\sigma\} & \sigma_* & -\sigma^* & & & -d_{F,G} & & & & \\
& (\cdot)\{\sigma^\prime\} & \sigma^\prime_* & & -\sigma^{\prime *} & & & -d_{F,G^\prime} & & & \\
& (\cdot)\{\tau\} &  &\tau_* & & -\tau^* & & & -d_{G,H} & & \\
& (\cdot)\{\tau^\prime \} & & &\tau^\prime_*  & -\tau^{\prime *} & & & & -d_{G^\prime ,H}  & \\
& \phi_{\sigma,\tau,\sigma^\prime,\tau^\prime}  & & & & & \tau_* & -\tau^\prime_* & \sigma^* & -\sigma^{\prime *} & d_{F,H}
 \end{array} \right] \]

\noindent where $\phi_{\sigma,\tau,\sigma^\prime,\tau^\prime} := (\cdot)\{\sigma, \tau\} -(\cdot)\{\sigma^\prime, \tau^\prime\}$
\end{example}

\begin{example}
Consider the pasting diagram given by four functors $F,G,H,K: {\cal A}\rightarrow {\cal B}$
and four natural transformations $\sigma:F \Rightarrow G$, $\tau:G \Rightarrow H$,
$\upsilon:H \Rightarrow K$ and $\chi:F \Rightarrow K$, with the obvious 0-, 1- and
2-cells and a single 3-cell enforcing the condition $\sigma \tau \upsilon= \chi$.

The deformation complex of the pasting diagram is then given by

\[ C^{\bullet+3}({\cal A}) \oplus C^{\bullet+3}({\cal B})\oplus C^{\bullet+2}(F) \oplus 
C^{\bullet+2}(G) \oplus C^{\bullet+2}(H) \oplus C^{\bullet+2}(K) \oplus C^{\bullet+1}(F,G) \] 
\[\oplus C^{\bullet+1}(G,H)
\oplus C^{\bullet+1}(H,K) \oplus C^{\bullet+1}(F,K) \oplus C^{\bullet}(F,K)\]

\noindent
with coboundary given by

\[ \left[
    \begin{array}{ccccccccccc}
-d_{\cal A} & & & & & & & & & &\\
& -d_{\cal B} & & & & & & & & & \\
-F_* & F^* & d_F & & & & & & & &\\
-G_* & G^* & & d_G & & & & & & &\\
-H_* & H^* & & & d_H & & & & & &\\
-K_* & K^* & & & & d_K & & & & & \\
& (\cdot)\{\sigma\} & \sigma_* & -\sigma^* & & & -d_{F,G} & & & & \\
& (\cdot)\{\tau\} & & \tau_* &  -\tau^* & & & -d_{G,H} & & & \\
& (\cdot)\{\upsilon\} & & &\upsilon_* & & -\upsilon^* & &  -d_{H,K} & & \\
& (\cdot)\{\chi \} & \chi_*  & & & &-\chi^* & & &  -d_{F,K}  & \\
& \psi_{\sigma,\tau,\upsilon}  & & & & & \tau_*\upsilon_* & \sigma^*\upsilon_* & (\sigma \tau)^* & -Id & d_{F,H}
 \end{array} \right] \]

\noindent where $\psi_{\sigma,\tau,\upsilon} := \upsilon_*(\cdot)\{\sigma, \tau\} + (\cdot)\{\sigma \tau, \upsilon\}$
\end{example}

Now, observe that in this example, we made a choice:  we computed the bottom row, and in paricular the map $\psi_{\sigma,\tau,\upsilon}$ by left-associating $\sigma$, $\tau$ and $\upsilon$.  Had we right-associated them, the bottom row, other than $\psi_{\sigma,\tau,\upsilon}$ would have remained the same (though it would have been more natural to write
$(\tau \upsilon)_*$, $\upsilon_* \sigma^*$ and $\tau^* \sigma^*$ as the names for the other maps).  The second entry, rather than being $\psi_{\sigma,\tau,\upsilon}$ would have been

\[ \tilde{\psi}_{\sigma,\tau,\upsilon} : = \sigma^*(\cdot )\{\tau, \upsilon\} + (\cdot )\{\sigma,
\tau \upsilon\} \].

Now, it is easy to see that the chain maps from 

\[ C^{\bullet+3}({\cal A}) \oplus C^{\bullet+3}({\cal B})\oplus C^{\bullet+2}(F) \oplus 
C^{\bullet+2}(G) \oplus C^{\bullet+2}(H) \oplus C^{\bullet+2}(K)  \] 
\[\oplus C^{\bullet+1}(F,G) \oplus C^{\bullet+1}(G,H)
\oplus C^{\bullet+1}(H,K) \oplus C^{\bullet+1}(F,K) \]

\noindent to $C^{\bullet}(F,H)$ given by the entries of the bottom row other than $d_{F,H}$
are not equal, and their difference has a single non-zero entry $\psi_{\sigma,\tau,\upsilon} - \tilde{\psi}_{\sigma,\tau,\upsilon}$.  There is however a contracting homotopy for this difference, given by the map with all by the second entry zero, and that entry given by
$(\cdot )\{\sigma, \tau, \upsilon\}$, due to the relationship between the brace and coboundary discovered in the classical case by Gerstenhaber and Voronov \cite{GV} and seen to apply in the categorical setting in \cite{Yetter}.  Thus in the homotopy category (or the derived category, if one prefers) our choice of left-association was a matter of indifference.

It is easy to compute more examples in which no 1-dimensional compositions occur, for instance a ``pillow'' with a pentagonal cross-section with a triple composition equal to a double composition, by combining the techniques of the last three examples.

Rather than write this out explicitly, we now turn to examples involving 1-dimensional composition:

\begin{example} \label{equalspost1comp}
Consider the pasting diagram consisting of three categories $\cal A$, $\cal B$ and $\cal C$, three functors $F,G:{\cal A} \rightarrow {\cal B}$ and $H:{\cal B}\rightarrow {\cal C}$, two natural transformations $\sigma:F \Rightarrow G$ and $\tau:H(F) \Rightarrow H(G)$ and a single 3-cell enforcing the condition that $H(\sigma) = \tau$.

The deformation complex of the pasting diagram is given by

\[ C^{\bullet+3}({\cal A}) \oplus C^{\bullet+3}({\cal B}) \oplus C^{\bullet+3}({\cal C}) \oplus C^{\bullet+2}(F) \oplus 
C^{\bullet+2}(G) \oplus C^{\bullet+2}(H)   \] 
\[ \oplus C^{\bullet+1}(F,G) \oplus C^{\bullet+1}(H(F),H(G))
 \oplus C^{\bullet}(H(F),H(G))\]

\noindent with coboundary given by

{\footnotesize
\[ \left[
    \begin{array}{ccccccccc}
-d_{\cal A} & & & & & & & & \\
& -d_{\cal B} & & & & & & &  \\
& & -d_{\cal C} & & & & & &  \\
-F_* & F^* & & d_F & & & & & \\
-G_* & G^* & & & d_G & & & & \\
& -H_* & H^* & & & d_H & & & \\
& (\cdot)\{\sigma\} & & \sigma_* & -\sigma^* & & -d_{F,G} & & \\
& & (\cdot)\{\tau\} & \tau_*(H_*) & -\tau^*(H_*) & -\tau^*(G^*) + \tau_*(F^*) & & -d_{(H(F),H(G))} & \\
& & & & & (\cdot)\{\sigma\} & H_* & -Id & d_{(H(F),H(G))} 
 \end{array} \right]  \] }

\end{example}

\begin{example} \label{equalspre1comp}
Consider the pasting diagram consisting of three categories $\cal A$, $\cal B$ and $\cal C$, three functors $F:{\cal A} \rightarrow {\cal B}$ and $G,H:{\cal B}\rightarrow {\cal C}$, two natural transformations $\sigma:G \Rightarrow H$ and $\tau:G(F) \Rightarrow H(F)$ and a single 3-cell enforcing the condition that $\sigma_F= \tau$.

The deformation complex of the pasting diagram is given by

\[ C^{\bullet+3}({\cal A}) \oplus C^{\bullet+3}({\cal B}) \oplus C^{\bullet+3}({\cal C}) \oplus C^{\bullet+2}(F) \oplus 
C^{\bullet+2}(G) \oplus C^{\bullet+2}(H)   \] 
\[ \oplus C^{\bullet+1}(F,G) \oplus C^{\bullet+1}(G(F),H(F))
 \oplus C^{\bullet}(G(F),H(G))\]

\noindent with coboundary given by

{\footnotesize
\[ \left[
    \begin{array}{ccccccccc}
-d_{\cal A} & & & & & & & & \\
& -d_{\cal B} & & & & & & &  \\
& & -d_{\cal C} & & & & & &  \\
-F_* & F^* & & d_F & & & & & \\
& -G_* & G^* & & d_G & & & & \\
& -H_* & H^* & & & d_H & & & \\
& & (\cdot)\{\sigma\} & & \sigma_* & -\sigma^* &  -d_{F,G} & & \\
& & (\cdot)\{\tau\} & \tau_*(G_*) & -\tau^*(H_*) & \tau_*(F^*) - \tau^*(F^*) & & -d_{(G(F),H(F))} & \\
& & & & &  & F_* & -Id & d_{(G(F),H(F))} 
 \end{array} \right]  \] }

\end{example}


\begin{example}
Consider the pasting diagram consisting of three categories $\cal A$, $\cal B$ and $\cal C$, four functors $F,G:{\cal A} \rightarrow {\cal B}$, $H:{\cal B}\rightarrow {\cal C}$ and $K:{\cal A}\rightarrow {\cal C}$, and three natural transformations $\sigma:F \Rightarrow G$, $\tau: GH \Rightarrow K$ and $\upsilon:FH \rightarrow K$, with the obvious 0-, 1- and 2- cells and a single 3-cell enforcing the condition that $H(\sigma)\tau = \upsilon$.

The deformation complex of the pasting diagram is given by

\[ C^{\bullet+3}({\cal A}) \oplus C^{\bullet+3}({\cal B}) \oplus C^{\bullet+3}({\cal C}) \oplus C^{\bullet+2}(F) \oplus 
C^{\bullet+2}(G) \oplus C^{\bullet+2}(H) \oplus C^{\bullet+2}(K)  \] 
\[ \oplus C^{\bullet+1}(F,G) \oplus C^{\bullet+1}(H(G),K)
\oplus C^{\bullet+1}(H(F),K)  \oplus C^{\bullet}(H(F),K)\]

\noindent with coboundary given by

{\footnotesize
\[ \left[
    \begin{array}{ccccccccccc}
-d_{\cal A} & & & & & & & & & &\\
& -d_{\cal B} & & & & & & & & & \\
& & -d_{\cal C} & & & & & & & & \\
-F_* & F^* & & d_F & & & & & & &\\
-G_* & G^* & & & d_G & & & & & &\\
& -H_* & H^* & & & d_H & & & & &\\
-K_* & & K^* & & & & d_K & & & & \\ 
& (\cdot)\{\sigma\} & & \sigma_* & -\sigma^* & & & -d_{F,G} & & &  \\
& &  (\cdot)\{\tau\} &  &\tau_* (H_*) & \tau_*(G^*) & -\tau^* & & -d_{H(G),K} & & \\
& &  (\cdot)\{\upsilon\} & \upsilon_* (H_*) &  & \upsilon_*(F^*) & -\upsilon^* & & & -d_{H(F),K} &  \\
& & \zeta_{H,\sigma, \tau} & & & \tau_*( \cdot\{\sigma\}) & & \tau_*(H_*) & [H_*(\sigma)]^* & -Id & d_{H(F),K} 
 \end{array} \right]  \] }

\noindent where $\zeta_{H,\sigma, \tau} := (\cdot)\{H_*(\sigma),\tau\}$

\end{example}

\begin{example}
Consider the pasting diagram consisting of three categories $\cal A$, $\cal B$ and $\cal C$, four functors $F:{\cal A} \rightarrow {\cal B}$, $G,H:{\cal B}\rightarrow {\cal C}$ and $K:{\cal A}\rightarrow {\cal C}$, and three natural transformations $\sigma:FH \Rightarrow K$, $\tau: G \Rightarrow H$ and $\upsilon:FG \rightarrow K$, with the obvious 0-, 1- and 2- cells and a single 3-cell enforcing the condition that $\tau_F\sigma = \upsilon$.

The deformation complex of the pasting diagram is given by

\[ C^{\bullet+3}({\cal A}) \oplus C^{\bullet+3}({\cal B}) \oplus C^{\bullet+3}({\cal C}) \oplus C^{\bullet+2}(F) \oplus 
C^{\bullet+2}(G) \oplus C^{\bullet+2}(H) \oplus C^{\bullet+2}(K)  \] 
\[ \oplus C^{\bullet+1}(H(F),K) \oplus C^{\bullet+1}(G,H)
\oplus C^{\bullet+1}(H(G),K)  \oplus C^{\bullet}(H(G),K)\]

\noindent with coboundary given by

{\footnotesize
\[ \left[
    \begin{array}{ccccccccccc}
-d_{\cal A} & & & & & & & & & &\\
& -d_{\cal B} & & & & & & & & & \\
& & -d_{\cal C} & & & & & & & & \\
-F_* & F^* & & d_F & & & & & & &\\
& -G_* & G^* & & d_G & & & & & &\\
& -H_* & H^* & & & d_H & & & & &\\
-K_* & & K^* & & & & d_K & & & & \\  
& &(\cdot)\{\sigma\} & \sigma_*(H_*) & &\sigma_*(F^*) & -\sigma^* & -d_{H(F),K} & & &  \\ 
& &  (\cdot)\{\tau\} &  &\tau_* &  -\tau^* & & & -d_{G,H} & & \\ 
& &  (\cdot)\{\upsilon\} & \upsilon_* (G_*) &  \upsilon_*(F^*) & & -\upsilon^* & & & -d_{H(G),K} &  \\
 & & \xi_{H,\sigma, \tau} & & & & & [F^*(\tau)]^* & \sigma_*(F^*) & -Id & d_{H(F),K} 
 \end{array} \right]  \]  }

\noindent where $ \xi_{F,\sigma, \tau} := (\cdot)\{F_*(\tau),\sigma\}$

\end{example}

\begin{example}
Consider the pasting diagram consisting of three categories $\cal A$, $\cal B$ and $\cal C$, four functors $F:{\cal A} \rightarrow {\cal C}$, $G,H:{\cal A}\rightarrow {\cal B}$ and $K:{\cal B}\rightarrow {\cal C}$, and three natural transformations $\sigma:F \Rightarrow GK$, $\tau: G \Rightarrow H$ and $\upsilon:F \rightarrow HK$, with the obvious 0-, 1- and 2- cells and a single 3-cell enforcing the condition that $\sigma K(\tau) = \upsilon$.

The deformation complex of the pasting diagram is given by

\[ C^{\bullet+3}({\cal A}) \oplus C^{\bullet+3}({\cal B}) \oplus C^{\bullet+3}({\cal C}) \oplus C^{\bullet+2}(F) \oplus 
C^{\bullet+2}(G) \oplus C^{\bullet+2}(H) \oplus C^{\bullet+2}(K)  \] 
\[ \oplus C^{\bullet+1}(F,K(G)) \oplus C^{\bullet+1}(G,H)
\oplus C^{\bullet+1}F, K(H))  \oplus C^{\bullet}(F, K(H))\]

\noindent with coboundary given by

{\footnotesize
\[ \left[
    \begin{array}{ccccccccccc}
-d_{\cal A} & & & & & & & & & &\\
& -d_{\cal B} & & & & & & & & & \\
& & -d_{\cal C} & & & & & & & & \\
-F_* & & F^* & d_F & & & & & & &\\
-G_* & G^* & & & d_G & & & & & &\\
-H_* & H^* & & & & d_H & & & & &\\
& -K_* & K^* & & & & d_K & & & & \\ 
& & (\cdot)\{\sigma\} & \sigma_* & -\sigma^*(K_*) & & -\sigma^*(G^*) & -d_{F,G} & & &  \\
&  (\cdot)\{\tau\} &  & & \tau_* & -\tau^* &  && -d_{H(G),K} & & \\  
& &  (\cdot)\{\upsilon\} & \upsilon_* &  & -\upsilon^*(K_*) & -\upsilon^*(H^*) & & & -d_{H(F),K} &  \\
& & \eta_{K,\sigma, \tau} & & & & \sigma^*( \cdot\{\tau\}) & [K_*(\tau)]_* & \sigma^*(K_*) & -Id & d_{H(F),K} 
 \end{array} \right]  \] }

\noindent where $\eta_{K,\sigma, \tau} := (\cdot)\{\sigma,K_*(\tau)\}$

\end{example}

\begin{example}
Consider the pasting diagram consisting of three categories $\cal A$, $\cal B$ and $\cal C$, four functors $F:{\cal A} \rightarrow {\cal C}$, $G:{\cal A}\rightarrow {\cal B}$ and $H,K:{\cal B}\rightarrow {\cal C}$, and three natural transformations $\sigma:F \Rightarrow GH$, $\tau: H \Rightarrow K$ and $\upsilon:F \rightarrow GK$, with the obvious 0-, 1- and 2- cells and a single 3-cell enforcing the condition that $\sigma(\tau_G) = \upsilon$.

The deformation complex of the pasting diagram is given by

\[ C^{\bullet+3}({\cal A}) \oplus C^{\bullet+3}({\cal B}) \oplus C^{\bullet+3}({\cal C}) \oplus C^{\bullet+2}(F) \oplus 
C^{\bullet+2}(G) \oplus C^{\bullet+2}(H) \oplus C^{\bullet+2}(K)  \] 
\[ \oplus C^{\bullet+1}(F,H(G)) \oplus C^{\bullet+1}(H,K)
\oplus C^{\bullet+1}F, K(G))  \oplus C^{\bullet}(F, K(G))\]

\noindent with coboundary given by

{\footnotesize
\[ \left[
    \begin{array}{ccccccccccc}
-d_{\cal A} & & & & & & & & & &\\
& -d_{\cal B} & & & & & & & & & \\
& & -d_{\cal C} & & & & & & & & \\
-F_* & & F^* & d_F & & & & & & &\\
-G_* & G^* & & & d_G & & & & & &\\ 
& -H_* & H^* & & & d_H & & & & &\\
& -K_* & K^* & & & & d_K & & & & \\ 
& & (\cdot)\{\sigma\} & \sigma_* & -\sigma^*(H_*) &  -\sigma^*(G^*) & & -d_{F,H(G)} & & &  \\
& & (\cdot)\{\tau\} &  & & \tau_* & -\tau^* &  & -d_{H,K} & & \\ 
& &  (\cdot)\{\upsilon\} & \upsilon_* &  -\upsilon^*(K_*) &  & -\upsilon^*(G^*) & & & -d_{F,K(G)} &  \\  
& & \kappa_{G,\sigma, \tau} & & & & & [G^*(\tau)]_* & \sigma^*(G^*) & -Id & d_{F,K(G)} 
 \end{array} \right]  \] }

\noindent where $\kappa_{G,\sigma, \tau} := (\cdot)\{\sigma,G^*(\tau)\}$

\end{example}


\section{Constructing Larger Deformation Complexes from Smaller}

First, observing that the map $\wp^\sigma_D$ is simply the identity in the case where the composable pasting diagram $D$ is simply the 2-cell $\sigma$ (with its sources and targets), and that the (dual) mapping cylinder on a map is the cone on the direct sum of the map and the identity on its target we have

\begin{proposition}
If $D$ is a pasting diagram with a single 3-cell the boundary of which has domain (resp. codomain) which is a composable pasting diagram $D^\prime$ and codomain (resp. domain) which is a single 2-cell, then ${\frak C}^\bullet(D)$ is quasi-isomorphic to the (dual) mapping cylinder of the map of Proposition 8.3 of \cite{Yetter} induced by the composition.
\end{proposition}

Observe this is a generalization of Proposition \ref{babyqi} and in special instances relates the (dual) mapping cylinder of the map in Proposition \ref{chainmapfrom1postcomp} (resp. \ref{chainmapfrom1precomp}) to the complex of Example \ref{equalspost1comp} (resp. \ref{equalspre1comp}).

The following is a trivial observation about the deformation complex constructed in Proposition 8.3 of \cite{Yetter} and the remarks following:

\begin{proposition} \label{uniontopushout}
If $D = D_1 \cup D_2$ is a pasting diagram which is the union of pasting diagrams $D_1$ and $D_2$, then the deformation complex ${\frak C}^\bullet(D)$ is the pushout of the induced inclusions of deformation complexes ${\frak C}(\iota_i)$ for  $\iota_i: D_1\cap D_2 \rightarrow D_i ,$ $i = 1,2$.
\end{proposition}

\section{The Polygonal Method}

In \cite{Shrestha} Shrestha developed a method of using polygons with edges labeled by arrow-valued operations to simultaneously encode cocycle conditions, the formulas for obstructions, and the condition that the next term of a deformation cobound the obstruction.  Suitable cell-decompositions of the surface of a 2-sphere into such polygons and ``trivial'' polygons which encode tautologous equalities in place of cocycle conditions and zero (as a difference of identical sums) in place of an obstruction, then provide a convenient method of proof for standard obstructions-are-cocycle results.

As an example, consider the cocycle conditions, formulas for obstructions, and cobounding conditions in the case of the deformation of composition, writing the undeformed composition as $\mu^{(0)}$ to match the notation for the degree $n$ deformation term as $\mu^{(n)}\epsilon^n$:

The cocycle condition is, of course,

\begin{eqnarray*}
 \lefteqn{\delta (\mu^{(1)})(a,b,c)} \\
 & = & \mu^{(0)}(\mu^{(1)}(a,b),c) -\mu^{(0)}(a,\mu^{(1)}(b,c)) + \mu^{(1)}(\mu^{(0)}(a,b),c)-\mu^{(1)}(a,\mu^{(0)}(b,c)) \\
 & = & \sum_{i+j = 1, i,j \in \{0,1\}} \mu^{(i)}(\mu^{(j)}(a,b),c) -\mu^{(i)}(a,\mu^{(j)}(b,c)) \\
 & = & 0
\end{eqnarray*}

The formula for the degree $n$ obstruction is

\[ \omega^{(n)}(a,b,c) = \sum_{i+j = n, i,j \in \{0,\ldots, n-1\}} \mu^{(i)}(\mu^{(j)}(a,b),c) -\mu^{(i)}(a,\mu^{(j)}(b,c)) \].

And the condition that $\mu^{(n)}$ cobounds $\omega^{(n)}$ is

\begin{eqnarray*}
 \omega^{(n)}(a,b,c) - \delta(\mu^{(n))}(a,b,c) & = & \sum_{i+j = n, i,j \in \{0,\ldots, n\}} \mu^{(i)}(\mu^{(j)}(a,b),c) -\mu^{(i)}(a,\mu^{(j)}(b,c)) \\
& = & 0 \\
\end{eqnarray*}

 Now consider the square, with oriented edges labeled by the compositions occuring in the expression of the associativity of composition:

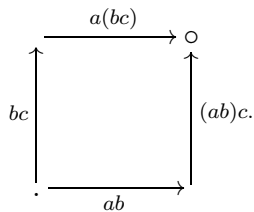
\[\xymatrixcolsep{4pc}\xymatrixrowsep{4pc}
\xymatrix{\ar[r]^{a(bc)} & \circ\\
.\ar[r]_{ab}\ar[u]^{bc} & \ar[u]_{(ab)c.}}.
\]\\

Each of the above formulas can be obtained from (or represented by) the square by adding
all possible terms obtained by assigning degrees to the edges in such a way that the degrees
are chosen from a particular set and the sum of degrees along each of the oriented parts of the boundary has a specified value, and taking the difference of the expressions thus obtained for each of the oriented parts of the boundary. 

Gerstenhaber's proof \cite{G} that obstructions to the deformation of an associative algebra are Hochschild cocycles can then be described in terms of the cube below by first noting that the formula for the coboundary of the degree $n$ obstruction can be written as a signed sum of the obstruction-type expressions (with the edge-degrees ranging from $0$ to $n-1$), five faces representing the usual terms of the coboundary, and one representing $0$, ``prolonged'' by pre- or post-composing with the degree $0$ label on another edge so that all of the expressions represent compositions of terms from the deformation of an iterated composition of all four maps, and thus, all representing parallel maps, can all be added.

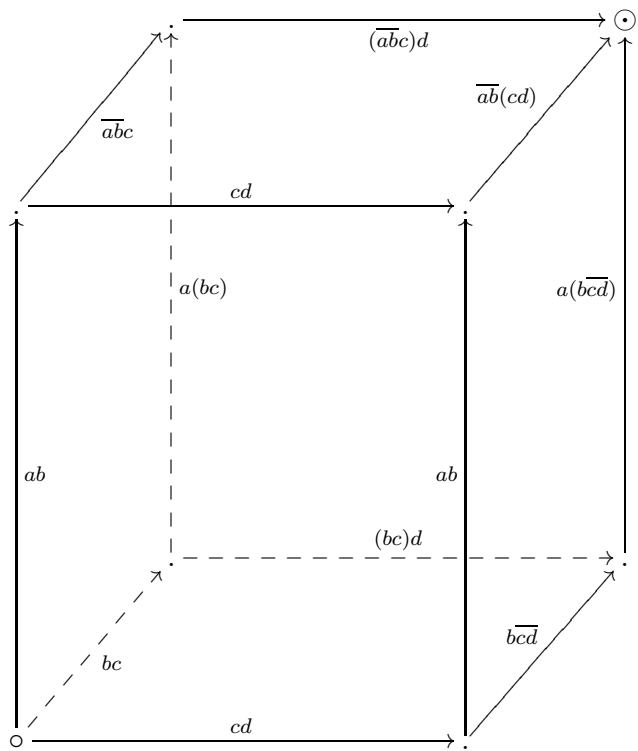
\[\xymatrixcolsep{4pc}\xymatrixrowsep{5pc}
\xymatrix{ & .\ar[rrr]_{(\overline{ab}c)d} & & & \odot \\
.\ar[ru]_{\overline{ab}c}\ar[rrr]^{cd}& &&.\ar[ru]^{\overline{ab}(cd)} & \\
&&&&\\
& .\ar@{-->}[uuu]_{a(bc)}\ar@{-->}[rrr]^{(bc)d} &  & & .\ar[uuu]^{a(b\overline{cd})}\\
\circ\ar[uuu]_{ab}\ar@{-->}[ru]_{bc}\ar[rrr]^{cd} & & & .\ar[ru]^{b\overline{cd}}\ar[uuu]^{ab}&
}
\]\\

The proof then proceeds by iteratively ``clearing'' edges shared by two squares with the same sign by rewriting terms involving that edge's label contributed by one square using the cobounding and cocycle condition from the other square sharing the edge until the expression is reduced to two copies of the obstruction-type expression on an equatorial hexagon, with opposite signs.

This observation, that Gerstenhaber's proof can be encoded by such a figure, then motivates a sequence of definitions and lemmas that allow for similar encoding of more general and complex proofs of the same sort.  Shrestha's technique \cite{Shrestha} is most quickly and rigorously described by labeling 1-cells of certain pasting schemes and computads (cf. \cite{Street}, \cite{Power}, \cite{Yetter}) with well-formed expressions in a particular (essentially) algebraic theory:

\begin{definition}
A {\em directed polygon} $P$ is a 2-computad with a single 2-cell, all of whose 0- and 1-cells lie in the boundary of the 2-cell.

A {\em whiskered polygon} $W$ is a 2-computad with a single 2-cell $P$, whose underlying 1-computad is the union of two 1-dimensional pasting schemes (directed paths), with the same 0-domain and 0-codomain, whose underlying cell complex is contractible.  We refer to the 1-dimensional pasting scheme consisting of the path from the 0-domain of  $W$ to the 0-domain of $P$ as the {\em domain whisker} (note: it may be simply a 0-cell), and the 1-dimensional pasting scheme consisting of the path from the 0-codomain of $P$ to the 0-codomain of $W$ as the {\em codomain whisker} (again it may simply be a 0-cell).

A {\em tiled sphere} is a 3-computad with a single 3-cell, all of whose 0-, 1- and 2-cells lie in the boundary of the 3-cell. 
\end{definition}

In applying the method in a given circumstance, one needs to label the edges of tiled spheres with arrow-valued operations from a theory associated to the pasting diagram (or more general structure).  For the present application the following suffice:

\begin{definition} \label{theory_of_diagram}
To any pasting diagram $D$, {\em theory of }$D$, ${\Bbb T}(D)$, is the  essentially algebraic theory with
types $O_{\cal C}$ and $A_{\cal C}$ for each category $\cal C$ in the diagram (the objects of $\cal C$ and the arrows of $\cal C$ respectively) and operations 

\begin{itemize}
\item $Id_{\cal C}(-)$ of arity $O_{\cal C}$ and type $A_{\cal C}$, 
\item $s_{\cal C}$ and $t_{\cal C}$ of arity $A_{\cal C}$ and type $O_{\cal C}$, and
\item $m_{\cal C}$ of arity  $A_{{\cal C} \;t}\!\! \times_s A_{\cal C}$ and type $A_{\cal C}$
\end{itemize}

\noindent for each category $\cal C$ in $D$;

\begin{itemize}
\item $F_O$ of arity $O_{\cal C}$ and type $O_{\cal D}$ and
\item $F_A$ of arity $A_{\cal C}$ and type $A_{\cal D}$
\end{itemize}

\noindent for each functor $F:\cal C \rightarrow \cal D$ in $D$; and 

\begin{itemize}
\item $\sigma$ of arity $O_{\cal C}$ and type $A_{\cal D}$ 
\end{itemize}

\noindent for each natural transformation $\sigma:F \Rightarrow G$ for $F,G:\cal C \rightarrow \cal D$ in $D$.

And, with axioms the equations expressing the axioms of categories for each 6-tuple $(O_{\cal C}, A_{\cal C}, Id_{\cal C}, s_{\cal C}, t_{\cal C}, m_{\cal C})$, the functoriality of each $(F_O, F_A)$ and the naturality of each $\sigma$.

\end{definition}

We term operations of type $A_{\cal C}$ for any $\cal C$ ``arrow-valued'' and those of type $O_{\cal C}$ for any $\cal C$ `` object-valued'', and refer to the model of ${\Bbb T}(D)$ given by $D$ as the ``tautologous model''.  In other applications of the method, as in \cite{Shrestha} ${\Bbb T}(D)$ may be replaced with an extension of the theory (for instance including operations of arity $A_{\cal C} \times A_{\cal C}$ and type $A_{{\cal C}\boxtimes {\cal C}}$ encoding monoidal product of two arrows).  We conjecture that broader generalizations, for instance to deformations of $n$-tuple categories and (weak) $n$- categories with a $k$-linear structure on their $n$-arrows, and to suitable pasting diagrams of these, or even to models of other sorts of essentially algebraic theories (cf. \cite{FS}) with appropriate linearizations of parts of the structure, can be described, but do not pursue this possibility here.

Notice a subtlety in the description of the axioms of ${\Bbb T}(D)$ in Definition \ref{theory_of_diagram}:  the only axioms are those inherited from the axioms of categories, functors and natural transformations.  In the Definition \ref{parallels}, on the contrary, all of the equations which hold in the diagram $D$ are enforced.

\begin{definition} \label{parallels}

The {\em theory of parallels} ${\Bbb P}(D)$ for a pasting diagram $D$ is the essentially algebraic theory with the same types and object valued operations as ${\Bbb T}(D)$ and with arrow-valued operations given by all set functions $\pi$ with the same domain and codomain as the instantiation of an arrow-valued operation $\omega$ of ${\Bbb T}(D)$ in the tautologous model and satisfying
$s(\pi) = s(\omega)$ and $t(\pi) = t(\omega)$  (we call $\pi$ {\em a parallel} for $\omega$), and all equations that hold among (iterated generalized) compositions of these functions as axioms.

\end{definition}

Notice in neither definition did we include any addition or scalar multiplication operations induced by the $k$-linear structure on the categories, even though we are only applying these constructions to pasting diagrams of $k$-linear categories, $k$-linear functors and natural transformations.   However, when applied to such a pasting diagram, the vector space structure on the hom-sets induces a vector space structure on the set of parallels for any arrow-valued operation $\omega$ of ${\Bbb T}(D)$.

Given a (countably infinite) stock of variables of each type in the theory it is evident what is meant by a well-formed formula (wff) of the theory.

We now make some technical definitions:

\begin{definition} An {\em equivalent} of a wff is any wff of the same type in the same variables such that for every instantiation of the variables in the two formulas, the values are equal.
\end{definition}

In what follows, for brevity we will refer to well-formed subformulas of a wff as wfss.  When considered as wfss of a fixed wff $W$, repetitions of equal wffs are considered distinct wfss.  With this convention, the following (a generality about wfss of a wff in any formal system) is immediate:

\begin{proposition}
The wfss of a wff $W$ form a partially ordered set $\Sigma(W)$ under $U \leq V$ when $U$ is a wfs of $V$, and the incidence diagram of $\Sigma(W)$ is a tree.  Thus $\Sigma(W)$ admits a natural number valued depth function $d$ for which $d(W) = 0$ and whenever $V$ covers $U$, $d(U) = d(V) + 1$.  The minimal elements (in the order-theoretic sense, not just those of maximum depth) for $\Sigma(W)$ are instances of variables.  We denote the subposet of non-minimal wfss of $W$ by $\Sigma(W)^\circ$
\end{proposition}

\begin{definition}
A {\em well-formed labeling} of a whiskered polygon (resp. tiled sphere) is an assignment $e \mapsto \lambda_e$ to each directed 1-cell of arrow-valued wff in ${\Bbb T}(D)$ for some pasting diagram (or an appropriate extension of this theory) with the properties 

\begin{itemize}
\item[WF0] Edges are not labeled with variables (i.e. all labels involve an operation in the theory).

\item[WF1] For each 1-pasting scheme (directed path) $e_1, e_2, \ldots e_n$ from the 0-domain of the whiskered polygon (resp. tiled sphere) the labels $\lambda_{e_1}, \ldots, \lambda_{e_n}$ can be iteratively replaced, beginning at the 0-codomain, with labels $\tilde{\lambda}_{e_1}, \ldots , \tilde{\lambda}_{e_n}$ such that 

\begin{itemize}
\item[WF1a] If the only proper wfss of $\lambda_{e_i}$ are variables then $\tilde{\lambda}_{e_i} = \lambda_{e_i}$.
\item[WF1b] For all $i$  $\tilde{\lambda}_{e_i}$  is equivalent to $\lambda_{e_i}$.
\item[WF1c] In the list of replacement labels $\tilde{\lambda}_{e_1}, \ldots , \tilde{\lambda}_{e_n}$, for each $\tilde{\lambda}_{e_i}$, if $\kappa$ is a wff occuring $m$ times as a non-minimal wfs of $\tilde{\lambda}_{e_i}$, then there exist $m$ distinct $j$'s with  $j < i$ such that $\tilde{\lambda}_{e_j} = \kappa$.
\end{itemize}

\item[WF2] If $e_1, e_2, \ldots e_n$ is a maximal directed path in a whiskered polygon (resp. tiled sphere), the replacement labels $\tilde{\lambda}_{e_1}, \ldots , \tilde{\lambda}_{e_n}$ are precisely the non-minimal wfss of $\tilde{\lambda}_{e_n}$.

\item[WF3] For every 2-cell, and every maximal extension of the 2-cell to a whiskered polygon, the labels of the final edges of the two paths from the 0-domain to the 0-codomain of the whiskered polygon are equivalents.  (Note this condition is vacuously true unless the 0-codomain of the 2-cell is the 0-codomain of the whiskered polygon.)

\end{itemize}
\end{definition}

The following is then immediate:

\begin{proposition}
For any maximal path $e_1, e_2, \ldots e_n$ in a whiskered polygon (resp. tiled sphere) the totally ordered set of wff's $\{\tilde{\lambda}_{e_1} \leq \ldots \leq \tilde{\lambda}_{e_n}\}$ is a totalization of the partially ordered set $\Sigma(\tilde{\lambda}_{e_n})^\circ$.
\end{proposition}

\begin{figure}
\begin{center}
\setlength{\unitlength}{4144sp}%
\begingroup\makeatletter\ifx\SetFigFontNFSS\undefined%
\gdef\SetFigFontNFSS#1#2#3#4#5{%
  \reset@font\fontsize{#1}{#2pt}%
  \fontfamily{#3}\fontseries{#4}\fontshape{#5}%
  \selectfont}%
\fi\endgroup%
\begin{picture}(3885,7414)(1816,-7874)
\thinlines
{\color[rgb]{0,0,0}\put(2926,-4201){\vector(-2,-3){750}}
}%
{\color[rgb]{0,0,0}\put(2206,-5356){\vector( 0,-1){1275}}
}%
{\color[rgb]{0,0,0}\put(2206,-6646){\vector( 4,-3){1112}}
}%
{\color[rgb]{0,0,0}\put(3316,-7456){\vector( 1, 0){1530}}
}%
{\color[rgb]{0,0,0}\put(4426,-4186){\vector( 4,-3){1080}}
}%
{\color[rgb]{0,0,0}\put(5536,-5041){\vector( 0,-1){1605}}
}%
{\color[rgb]{0,0,0}\put(5558,-6669){\vector(-1,-1){735}}
}%
{\color[rgb]{0,0,0}\put(2926,-4111){\vector(-2, 3){750}}
}%
{\color[rgb]{0,0,0}\put(2206,-2956){\vector( 0, 1){1275}}
}%
{\color[rgb]{0,0,0}\put(2206,-1666){\vector( 4, 3){1112}}
}%
{\color[rgb]{0,0,0}\put(3316,-856){\vector( 1, 0){1530}}
}%
{\color[rgb]{0,0,0}\put(4426,-4126){\vector( 4, 3){1080}}
}%
{\color[rgb]{0,0,0}\put(5536,-3271){\vector( 0, 1){1605}}
}%
{\color[rgb]{0,0,0}\put(5558,-1643){\vector(-1, 1){735}}
}%
{\color[rgb]{0,0,0}\put(3016,-4156){\vector( 1, 0){1350}}
}%
{\color[rgb]{0,0,0}\put(3001,-4261){\vector( 0,-1){1350}}
}%
{\color[rgb]{0,0,0}\put(2963,-5657){\vector(-3,-4){642}}
}%
{\color[rgb]{0,0,0}\put(3031,-4231){\vector( 1,-1){765}}
}%
{\color[rgb]{0,0,0}\put(3841,-5011){\vector( 1, 0){1560}}
}%
{\color[rgb]{0,0,0}\put(3061,-5641){\vector( 3,-2){675}}
}%
{\color[rgb]{0,0,0}\put(3751,-5056){\vector( 0,-1){975}}
}%
{\color[rgb]{0,0,0}\put(3796,-6091){\vector( 3,-4){945}}
}%
{\color[rgb]{0,0,0}\put(2251,-2986){\vector( 1, 0){1320}}
}%
{\color[rgb]{0,0,0}\put(3946,-1696){\vector( 1, 0){1485}}
}%
{\color[rgb]{0,0,0}\put(3871,-1726){\vector(-2, 3){550}}
}%
{\color[rgb]{0,0,0}\put(3608,-2913){\vector( 1, 4){285}}
}%
{\color[rgb]{0,0,0}\put(4321,-4111){\vector(-2, 3){750}}
}%
\put(1831,-2326){\makebox(0,0)[lb]{\smash{{\SetFigFontNFSS{12}{14.4}{\rmdefault}{\mddefault}{\updefault}{\color[rgb]{0,0,0}$ab$}%
}}}}
\put(2176,-1126){\makebox(0,0)[lb]{\smash{{\SetFigFontNFSS{12}{14.4}{\rmdefault}{\mddefault}{\updefault}{\color[rgb]{0,0,0}$F(ab)$}%
}}}}
\put(2641,-2806){\makebox(0,0)[lb]{\smash{{\SetFigFontNFSS{12}{14.4}{\rmdefault}{\mddefault}{\updefault}{\color[rgb]{0,0,0}$F(a)$}%
}}}}
\put(3871,-2371){\makebox(0,0)[lb]{\smash{{\SetFigFontNFSS{12}{14.4}{\rmdefault}{\mddefault}{\updefault}{\color[rgb]{0,0,0}$F(b)$}%
}}}}
\put(3706,-1246){\makebox(0,0)[lb]{\smash{{\SetFigFontNFSS{12}{14.4}{\rmdefault}{\mddefault}{\updefault}{\color[rgb]{0,0,0}$F(a)F(b)$}%
}}}}
\put(3661,-631){\makebox(0,0)[lb]{\smash{{\SetFigFontNFSS{12}{14.4}{\rmdefault}{\mddefault}{\updefault}{\color[rgb]{0,0,0}$F(ab)F(c)$}%
}}}}
\put(2161,-3706){\makebox(0,0)[lb]{\smash{{\SetFigFontNFSS{12}{14.4}{\rmdefault}{\mddefault}{\updefault}{\color[rgb]{0,0,0}$F(c)$}%
}}}}
\put(5296,-1096){\makebox(0,0)[lb]{\smash{{\SetFigFontNFSS{12}{14.4}{\rmdefault}{\mddefault}{\updefault}{\color[rgb]{0,0,0}$F(a)F(bc)$}%
}}}}
\put(3991,-3361){\makebox(0,0)[lb]{\smash{{\SetFigFontNFSS{12}{14.4}{\rmdefault}{\mddefault}{\updefault}{\color[rgb]{0,0,0}$F(c)$}%
}}}}
\put(3241,-3976){\makebox(0,0)[lb]{\smash{{\SetFigFontNFSS{12}{14.4}{\rmdefault}{\mddefault}{\updefault}{\color[rgb]{0,0,0}$F(a)$}%
}}}}
\put(5041,-3886){\makebox(0,0)[lb]{\smash{{\SetFigFontNFSS{12}{14.4}{\rmdefault}{\mddefault}{\updefault}{\color[rgb]{0,0,0}$bc$}%
}}}}
\put(5686,-2566){\makebox(0,0)[lb]{\smash{{\SetFigFontNFSS{12}{14.4}{\rmdefault}{\mddefault}{\updefault}{\color[rgb]{0,0,0}$F(bc)$}%
}}}}
\put(2071,-4621){\makebox(0,0)[lb]{\smash{{\SetFigFontNFSS{12}{14.4}{\rmdefault}{\mddefault}{\updefault}{\color[rgb]{0,0,0}$F(c)$}%
}}}}
\put(5071,-4501){\makebox(0,0)[lb]{\smash{{\SetFigFontNFSS{12}{14.4}{\rmdefault}{\mddefault}{\updefault}{\color[rgb]{0,0,0}$bc$}%
}}}}
\put(5671,-6061){\makebox(0,0)[lb]{\smash{{\SetFigFontNFSS{12}{14.4}{\rmdefault}{\mddefault}{\updefault}{\color[rgb]{0,0,0}$F(bc)$}%
}}}}
\put(1831,-6046){\makebox(0,0)[lb]{\smash{{\SetFigFontNFSS{12}{14.4}{\rmdefault}{\mddefault}{\updefault}{\color[rgb]{0,0,0}$ab$}%
}}}}
\put(2161,-7411){\makebox(0,0)[lb]{\smash{{\SetFigFontNFSS{12}{14.4}{\rmdefault}{\mddefault}{\updefault}{\color[rgb]{0,0,0}$F(ab)$}%
}}}}
\put(3541,-7801){\makebox(0,0)[lb]{\smash{{\SetFigFontNFSS{12}{14.4}{\rmdefault}{\mddefault}{\updefault}{\color[rgb]{0,0,0}$F(ab)F(c)$}%
}}}}
\put(5356,-7261){\makebox(0,0)[lb]{\smash{{\SetFigFontNFSS{12}{14.4}{\rmdefault}{\mddefault}{\updefault}{\color[rgb]{0,0,0}$F(a)F(bc)$}%
}}}}
\put(2746,-5131){\makebox(0,0)[lb]{\smash{{\SetFigFontNFSS{12}{14.4}{\rmdefault}{\mddefault}{\updefault}{\color[rgb]{0,0,0}$ab$}%
}}}}
\put(4426,-5281){\makebox(0,0)[lb]{\smash{{\SetFigFontNFSS{12}{14.4}{\rmdefault}{\mddefault}{\updefault}{\color[rgb]{0,0,0}$F(a)$}%
}}}}
\put(2581,-6391){\makebox(0,0)[lb]{\smash{{\SetFigFontNFSS{12}{14.4}{\rmdefault}{\mddefault}{\updefault}{\color[rgb]{0,0,0}$F(c)$}%
}}}}
\put(3496,-4516){\makebox(0,0)[lb]{\smash{{\SetFigFontNFSS{12}{14.4}{\rmdefault}{\mddefault}{\updefault}{\color[rgb]{0,0,0}$bc$}%
}}}}
\put(3841,-5581){\makebox(0,0)[lb]{\smash{{\SetFigFontNFSS{12}{14.4}{\rmdefault}{\mddefault}{\updefault}{\color[rgb]{0,0,0}$a(bc)$}%
}}}}
\put(2971,-6121){\makebox(0,0)[lb]{\smash{{\SetFigFontNFSS{12}{14.4}{\rmdefault}{\mddefault}{\updefault}{\color[rgb]{0,0,0}$(ab)c$}%
}}}}
\put(4201,-6406){\makebox(0,0)[lb]{\smash{{\SetFigFontNFSS{12}{14.4}{\rmdefault}{\mddefault}{\updefault}{\color[rgb]{0,0,0}$F(a(bc))$}%
}}}}
\put(2416,-5491){\makebox(0,0)[lb]{\smash{{\SetFigFontNFSS{12}{14.4}{\rmdefault}{\mddefault}{\updefault}{\color[rgb]{0,0,0}$T$}%
}}}}
\put(2956,-3421){\makebox(0,0)[lb]{\smash{{\SetFigFontNFSS{12}{14.4}{\rmdefault}{\mddefault}{\updefault}{\color[rgb]{0,0,0}$T$}%
}}}}
\put(4111,-4606){\makebox(0,0)[lb]{\smash{{\SetFigFontNFSS{12}{14.4}{\rmdefault}{\mddefault}{\updefault}{\color[rgb]{0,0,0}$T$}%
}}}}
\put(4351,-1981){\makebox(0,0)[lb]{\smash{{\SetFigFontNFSS{12}{14.4}{\rmdefault}{\mddefault}{\updefault}{\color[rgb]{0,0,0}$F(b)F(c)$}%
}}}}
\end{picture}%

\end{center}
\caption{A well-formed labeling of a tiled sphere associated to the deformation of a functor \label{functor_obstructions_are_cocycles}}
\end{figure}

An example of a well-formed labeling of a tiled sphere is given in Figure \ref{functor_obstructions_are_cocycles}. (The boundaries of the two octogons bordering the unbounded region should be identified along the edges with the same labels, and the 3-cell inserted to fill the resulting topological sphere.)   This particular example arises in applying the method to show that the obstructions to deforming (the arrow part of) a functor, when the compositions in its source and target categories are also being deformed, are cocycles.  Note that in the labeling, at one point notation has been abused and $F(abc)$ has been used in place of either of the equivalent wffs $F((ab)c)$ or $F(a(bc))$.  It is also easy to see along which paths the label-replacement axiom WF1 results in non-trival replacements of labels (as, for instance along any path ending in the top right-most edge in which $F(b)F(c)$ occurs, rather than $F(bc)$:  the label of the top right-most edge must be replaced with $F(a)(F(b)F(c))$).

Now each arrow-valued operation is a part of the algebraic structure which is subject to infinitesimal deformation, by being replaced by a formal power-series (or polynomial in $\epsilon$ with $\epsilon^n = 0$ for some $n$) whose coefficients are arrow-valued operations of the same type and arity (e.g. the arrow part of a functor $F:\cal C \rightarrow \cal D$ is replaced with $\tilde{F} = 
\sum F^{(n)}\epsilon^n$ where for each $a \in Arr(\cal C)$, $F^{(n)}(a)$ is a map from $s(F(a))$ to $t(F(a))$).  From this, we abstract

\begin{definition}  Consider any set of arrow-valued operations $\Bbb O$ in ${\Bbb T}(D)$ for some $D$.
Fix $n \in \Bbb N \cup \{\infty\}$.  To each arrow-valued operation $\psi \in \Bbb O$ associate a sequence of parallels $\psi^{(k)}$ in the sense of Definition \ref{parallels} (truncated at $k = n$ if $n < \infty$) with
$\psi^{(0)} = \psi$.
We call $\psi^{(k)}$ {\em the degree $k$ parallel of $\psi$}, and a choice $\Bbb O_n$ of such a sequence for every operation in $\Bbb O$,  {\em a degree $n$ family of parallels for $\Bbb O$}
\end{definition}

Now, for any well-formed labeling of a whiskered polygon or tiled sphere with, let $\Bbb O$ be the set of all arrow-valued operations occuring in the labels of the edges, closed under equivalence.  Every family of parallels for $\Bbb O$, then gives rise to many labelings of the maximal paths of the whiskered polygon or tiled sphere by wffs from ${\Bbb P}(D)$ by replacing the labels along the path using WF1, then for each edge, chosing a degree $k$ and replacing the last-applied operation $\psi$ in the label on that edge with $\psi^{(k)}$ in the label for that edge and all edges later in the path for which the label on the given edge as a wfs.

In particular, any such choice of degrees for each edge along a path from the 0-domain to the 0-codomain creates a new wff in ${\Bbb P}(D)$.  Suppose the sequence of labels on a maximal path was $f_1, f_2, \ldots , f_n$, where each $f_i$ denotes a well-formed expression all of whose iterated inputs have an equivalent among the labels earlier in the sequence, and by abuse of notion also that well-formed expression's last-applied operation. A choice of degrees $i_1, i_2, \ldots , i_n$ then produces a new sequence of well-formed expression $f_1^{(i_1)}, f_2^{(i_2)}, \ldots , f_n^{(i_n)}$ in which each instance of an operation has been replaced with its parallel of the chosen degree in that and all later edge-labels.  Note:in general these new labelings are not well-formed labelings.

At this point, recall that all of this is taking place in a linear setting, so that parallel arrows (and thus parallels) can be added.  We will now define several different expressions which a system of parallels associates to a whiskered polygon.  

\begin{definition} \label{conditions_and_obstructions}

For any well-formed labeling of a whiskered polygon $W$, $\Bbb O$ the set of operations occuring in the replacement labelings of both paths, and $\Bbb O_n$ a family of parallels for $\Bbb O$, let $f_1,\ldots, f_k$ (resp. $g_1,\ldots, g_\ell$) be the replacement labeling on the maximal path which traverses domain (resp. codomain) of the 2-cell. 

The {\em cocycle-type condition} associated to $W$ is the equation

\[ \sum_{\stackrel{i_1,\ldots, i_k  \in  \{0,1\} }{ i_1 + \ldots + i_k =  1}}
 f_1^{(i_1)}, f_2^{(i_2)}, \ldots , f_n^{(i_k)}  - 
\sum_{\stackrel{ j_1,\ldots, j_\ell \in \{0,1\}}{ j_1 + \ldots + j_\ell  =  1}}
 g_1^{(j_1)}, g_2^{(j_2)}, \ldots , g_\ell^{(j_\ell)} = 0. \]

The {\em $m^{th}$ order obstruction-type expression} associated to $W$ is the expression

\[ \sum_{\stackrel{i_1,\ldots, i_k  \in  \{0,\ldots, m-1\} }{ i_1 + \ldots + i_k =  m}}
 f_1^{(i_1)}, f_2^{(i_2)}, \ldots , f_n^{(i_k)}  - 
\sum_{\stackrel{ j_1,\ldots, j_\ell \in \{0,\ldots, m-1\}}{ j_1 + \ldots + j_\ell  =  m}}
 g_1^{(j_1)}, g_2^{(j_2)}, \ldots , g_\ell^{(j_\ell)} . \]

And, the {\em $m^{th}$ order cobounding-type condition} associated to $W$ is the expression

\[ \sum_{\stackrel{i_1,\ldots, i_k  \in  \{0,\ldots, m\} }{ i_1 + \ldots + i_k =  m}}
 f_1^{(i_1)}, f_2^{(i_2)}, \ldots , f_n^{(i_k)}  - 
\sum_{\stackrel{ j_1,\ldots, j_\ell \in \{0,\ldots, m\}}{ j_1 + \ldots + j_\ell  =  m}}
 g_1^{(j_1)}, g_2^{(j_2)}, \ldots , g_\ell^{(j_\ell)} = 0 . \]

\end{definition}

In each case, the string of parallels should be understood as naming the well-formed experssion obtained by the iterated substitution of the named parallel for the corresponding operation in the well-formed expression labeling the last edge of the path.  

The following then provides the basis for Shrestha's method:

\begin{proposition} \label{elementary}
For any well-formed labeling of a tiled sphere $T$ with 3-cell $C$, $\Bbb O$ the set of operations occuring in the replacement labelings of all paths, and $\Bbb O_n$ a family of parallels for $\Bbb O$, the domain (resp. codomain) of $C$ can be expressed as the union of whiskered polygons with the same 0-domain and 0-codomain as $T$, one with each 2-cell in $dom(C)$ (resp. $cod(C)$) as its 2-cell in such a way that the sum of the $m^{th}$ order obstruction-type expressions for the whiskered polygons is the $m^{th}$ order obstruction-type expression for the (whiskered) polygon consisting of a single 2-cell and the union of the 1-pasting schemes $dom(dom(C)) = dom(cod(C))$ and $cod(dom(C)) = cod(cod(C))$.
\end{proposition}

\begin{proof}
Once the combinatorial structure of Power's proof of the the uniqueness of pasting compositions \cite{Power} is recalled, the result is immediate -- sums cancel in pairs leaving only the difference giving the desired $m^{th}$ order obstruction-type expression.
\end{proof}

Unfortunately, as it stands, the result is not immediately applicable.  Recall the cube encoding Gerstenhaber's proof.  Any face is part of a whiskered polygon (with one whisker having one edge, and the other having none).  The $m^{th}$ order obstruction is given by 

\[  \sum_{\stackrel{i,j  \in  \{0,\ldots, m-1\} }{ i + j =  m}} \mu^{(i)}(\mu^{(j)}(a,b), c) - 
 \sum_{\stackrel{i,j  \in  \{0,\ldots, m-1\} }{ i + j =  m}} \mu^{(i)}(a, \mu^{(j)}(b,c)) \]

The terms in its coboundary are {\em not} $m^{th}$ order obstruction-type expressions associated to the faces of their cubes with their whiskers.  The $m^{th}$ order obstruction type expressions are, instead, instances of $\mu^{(0)} = \mu$ with the obstruction and a variable as inputs, or instances of the obstruction with one of its variables replaced with $\mu^{(0)} = \mu$ applied to two variables.  What then is the relationship between actual  $m^{th}$ order obstructions and $m^{th}$ order obstruction-type expressions?

If one considers one of the whiskered polygons in the example, it is easy to see that the terms in the $m^{th}$ order obstruction-type expression which do not correspond to terms from the coboundary of the obstruction have a label of positive degree on the edge of the whisker.  If we fix the label on the whisker edge to be of degree $d > 0$, the terms with this label are then the condition that $\mu^{(m-d)}$ cobound the degree $m-d$ obstruction (or the cocycle condition if $m-d = 1$) with $\mu^{(d)}$ of two variables as argument (resp. used as input to an instance of $\mu^{(d)}$) when the non-trivial whisker is the domain (resp. codomain) whisker.

In this case, provided the system of parallels is describing the terms of an associative deformation, the $m^{th}$ order obstruction-type expression is thus equal to the corresponding term in the coboundary of the $m^{th}$ order obstruction, since its extra terms all vanish by the cocycle and cobounding conditions.

To imitate this in general, we need conditions depending only on the labels on the domain and codomain of the 2-cell in a whiskered polygon which ensures that the cocycle- and cobounding-type conditions in all lower degrees hold.

As in Definition \ref{conditions_and_obstructions} consider a whiskered polygon $W$ with 2-cell $P$ and let $f_1,\ldots, f_k$ (resp. $g_1,\ldots, g_\ell$) be the replacement labeling on the maximal path which traverses domain (resp. codomain) of $P$.  Suppose the first $s$ edges lie in the domain whisker.  In this case, for $i \leq s$, $f_i$ and $g_i$ are equivalents, and the orderings on the edge labels of the domain whisker induced by restricting the partial orderings on $\Sigma(f_k)$ and $\Sigma(g_\ell)$ coincide.  Let $(D, \leq)$ be the resulting partially ordered set of equivalence classes of wffs.

\begin{definition} \label{strong_vanishing}

Given a well-formed labeling of a whiskered polygon $W$ with 2-cell $P$, let $\Bbb O$ be the set of operations occuring in the replacement labelings of both paths, and $\Bbb O_n$ a family of parallels for $\Bbb O$.  Recall that ${\Bbb P}(W)$ is the theory of parallels for $W$ (and thus includes the elements of $\Bbb O_n$. Let $s$ and $(D, \leq)$ be as in the discussion above.  And suppose the last $t$ edges of each path lie in the codomain whisker.

The {\em $m^{th}$ order strong vanishing condition} associated to $W$ (for $m \geq 0$) is the condition that

\begin{eqnarray*}
\forall \Phi_1,\ldots \Phi_t  \sum_{\stackrel{i_{s+1},\ldots, i_{k-t}  \in  \{0,\ldots, m\} }{ i_{s+1} + \ldots + i_{k-t} =  m}}
\hat{ f}_{s+1}^{(i_{s+1})}, \hat{f}_{s+2}^{(i_{s+2})}, 
\ldots , \hat{f}_{k-t}^{(i_{k-t})},\Phi_1, \ldots \Phi_t  - & &\\
\sum_{\stackrel{ j_{s+1},\ldots, j_{\ell-t} \in \{0,\ldots, m\}}{ j_{s+1} + \ldots + j_{\ell-t}  =  m}}
 \hat{g}_{s+1}^{(j_{s+1})}, \hat{g}_{s+2}^{(j_{s+2})}, \ldots , \hat{g}_{\ell-t}^{(j_{\ell-t})},\Phi_1, \ldots \Phi_t & = & 0. 
\end{eqnarray*}

\noindent where the hatting of wffs indicates the result of the following process:  for each maximal element of $(D, \leq)$ select a variable of the same type which does not occur in among the variables used in the labeling; the hatted wff is the result of replacing all wfss in that equivalence class with the corresponding variable; and the universal quantification ranges over all parallels in ${\Bbb P}(W)$ to the labels on the corresponding edge of the codomain whisker.

\end{definition}

We then have

\begin{proposition} \label{strong_works}
For any well-formed labeling of a whiskered polygon $W$ with operations ${\Bbb O}$ and a family of parallels ${\Bbb O}_n$, the $1^{st}$ order strong vanishing condition implies the cocycle-type condition, the $p^{th}$ order strong vanishing condition implies the $p^{th}$ order cobounding-type condition, and, moreover, if the  $p^{th}$ order vanishing conditions hold for all $0 \leq p < m$, then the $m^{th}$ order coboundary-type expression is equal to 

\begin{eqnarray*}
 \sum_{\stackrel{i_{s+1},\ldots, i_{k-t}  \in  \{0,\ldots, m-1\} }{ i_{s+1} + \ldots + i_{k-t} =  m}}
f_1,\ldots ,f_s, f_{s+1}^{(i_{s+1})}, f_{s+2}^{(i_{s+2})}, \ldots , f_{k-t}^{(i_{k-t})}, f_{k-t+1}, \ldots, f_k  -  & &\\
\sum_{\stackrel{ j_{s+1},\ldots, j_{\ell-t} \in \{0,\ldots, m-1\}}{ j_{s+1} + \ldots + j_{\ell-t}  =  m}}
g_1,\ldots ,g_s, g_{s+1}^{(j_{s+1})}, g_{s+2}^{(j_{s+2})}, \ldots , g_{\ell-t}^{(j_{\ell-t})}, g_{\ell-t+1}, \ldots g_\ell . & &  
\end{eqnarray*}

\noindent where $s$ and $t$ are as in the previous definition and the meanings of the sequences of wffs are as in Definition \ref{conditions_and_obstructions}.
\end{proposition}

\begin{proof}
The cocycle-type condition simply an instantiation of the $1^{st}$ order strong vanishing condition, while the $p^{th}$ order cobounding-type condition is the sum of all instantiations of the $p^{th}$ order strong vanishing condition ranging over all choices of parallels for the labels of edges in the whiskers. 

For the last statement, notice that the terms in the $m^{th}$ order coboundary-type expression which are not represented in the expression of the proposition all have at least one label on an edge of one of the whiskers which is of positive degree.  These terms can be partitioned into subsets which according to the degrees of the labels in the whiskers.  For each choice of degrees for the labels in the whiskers, the terms are an instatiation of the strong cocycle-type condition, and thus add to zero.
\end{proof} 

Two sorts of whiskered polygons satisfying the strong vanishing conditions arise in practice:  ``non-trivial'' whiskered polygons in which the strong vanishing conditions follow from the well-formedness of the labeling and deformation theoretic cocycle and cobounding conditions satisfied by the parallels of the labels in the boundary of the 2-cell, and ``trivial'' whiskered polygons in which the sets of edge labels on the domain and codomain of the 2-cell differ only by changing the choice of totalization of the partial order on wfss of the final label in the path and substitution of equivalent wffs.

We formalize this in the following propositions:

\begin{proposition} \label{nontrivial}
Given a well-formed labeling of a whiskered polygon $W$ with 2-cell $P$, let $\Bbb O$ be the set of operations occuring in the replacement labelings of both paths, and $\Bbb O_n$ a family of parallels for $\Bbb O$.  Let $s$, $t$, $k$, $\ell$ and $(D, \leq)$ and the hatting of wffs be as Definition \ref{strong_vanishing}.

If $\hat{f}_{s+1}, \ldots,  \hat{f}_{k-t}$ and $\hat{g}_{s+1}, \ldots , \hat{g}_{\ell-t}$ are a well-formed labeling of the directed polygon $P$, then $W$ satisfies the $0^{th}$ order strong vanishing condition.  If, moreover, this labeling of $P$ and the system of parallels satisfy the cocycle-type condition (resp. the $m^{th}$ order cobounding-type condition) then the well-formed labeling of $W$ and system of parallels satisfies the $1^{st}$ order (resp. $m^{th}$ order strong vanishing condition).
\end{proposition}

\begin{proof}
Notice first that the condition that the labeling of $P$ be well-formed means that the wff labeling the two edges incident with the 0-codomain of $P$ are equivalents, and thus their difference (as an arrow-valued operation) vanishes.  The well-formedness may thus be viewed as a zeroth order analogue of the cocycle-type and cobounding-type relations.

The proposition is immediate once it is observed that at each order the strong vanishing condition on $W$ is simply the result of applying the parallels to labels on the codomain whisker to the terms of the expression of that order which vanishes on $P$.
\end{proof}

In practice the labels on $P$ come from an equational condition (associativity, functoriality, or naturality in the present work) and the cocycle-type and cobounding-type conditions are actual cocycle and cobounding conditions derived from the requirement that deformations preserve the equational condition, while the whiskers arise in taking coboundaries of the corresponding obstruction.

\begin{proposition} \label{trivial}
Given a well-formed labeling of a whiskered polygon $W$ with 2-cell $P$, let $s$, $t$, $k$, $\ell$, and $(D, \leq)$ and the hatting of wffs be as Definition \ref{strong_vanishing}.

If $k = \ell$ and $\hat{f}_{s+1}, \ldots,  \hat{f}_{k-t}$ is obtained from $\hat{g}_{s+1}, \ldots , \hat{g}_{\ell-t}$ by changing the totalization of the partial ordering on the wfss of the label on the edge(s) incident with the 0-codomain of $W$ and replacing wffs with equivalents, then the well-formed label of $W$ satisfies the strong vanishing conditions of all orders for any system of parallels.
\end{proposition}

\begin{proof}
In this case for any choice of parallels labels on edges of the codomain whisker, and of degrees in a system of parallels for the corresponding $\hat{f}_{\sigma(i)}$ and $\hat{g}_i$, the resulting expressions for the two paths are equivalent and thus their difference is zero, and summing over all such choices of total order $p$ gives the $p^{th}$ order strong vanishing condition for $W$.
\end{proof}

We refer to whiskered polygons equipped with a labeling satifying the hypotheses of Proposition \ref{trivial} as trivial whiskered polygons.  

Finally, we need to describe in general the relationship between the cohomological description of infinitesimal deformations and the evident expression of the same data in terms of systems of parallels and the various expressions given by labeling of (whiskered) polygons and tiled spheres.

\begin{definition}
If $D$ is a $k$-linear pasting diagram, ${\Bbb T}(D)$ its theory, or an extension thereof with the same types, its deformation theory is {\em polygonizable} if the cochain group in which cocycles specify first order deformations admits a direct sum decomposition indexed a family of equational axioms of the theory each of which can be expressed as the vanishing of the difference of the values of the paths in a well-formed labeling of a directed polygon (we call such a polygon equipped with its well-formed labeling an {\em axiomatic polygon}), and which, moreover satisfy

\begin{itemize}
\item[P1] The vanishing of each direct summand of the cocycle is precisely the cocycle-type condition associated to the axiomatic polygon indexing the direct summand.
\item[P2] Each direct summand of the $m^{th}$ order obstruction is the obsturuction-type expression associated to the axiomatic polygon indexing the direct summand.
\item[P3] The cobounding condition for the extension of a deformation to the next degree is the direct sum of the cobounding-type conditions associated to the axiomatic polygons.
\item[P4] And the coboundary of an obstruction admits a direct sum decomposition in which each direct summand is a signed sum of expressions specified by labeling whiskered polygons with the polygon labeled by the obstruction, and the whiskers labeled by operations of ${\Bbb T}(D)$ (degree 0 labels in the system of parallels naming the deformation terms).
\end{itemize}

\end{definition}

We are now in a position to state a theorem which encapsulates Shrestha's polygonal method for our purposes:

\begin{theorem}
If ${\Bbb T}(D)$ is the theory of a pasting diagram, or an extension thereof, and admits a polygonizable deformation theory, then all obstructions are cocycles, provided for each direct summand of the coboundary of the obstruction, there exists a tiled sphere, and a well-formed labeling of the tiled sphere with the properties

\begin{itemize}
\item[S1] Each whiskered polygon naming a summand in the signed sum of P4 occurs exactly once in the tiled sphere -- either in the domain (resp. codomain) if its sign is positive (resp. negative), or with its domain and codomain swapped and in the codomain (resp. domain) if its sign is positive (resp. negative).
\item[S2] Every 2-cell which is not part of the whiskered polygons of S1 is trivial in the sense of
Proposition \ref{trivial}
\end{itemize}

\end{theorem}

\begin{proof}  This follows immediately from Propositions \ref{elementary}, \ref{strong_works}, \ref{nontrivial} and \ref{trivial}.
\end{proof}

Figures 1 through 6 then establish the following:

\begin{theorem} \label{singlecomposition}
If $D$ is a pasting diagram with a single instance of composing two natural transformations or of composing a natural transformation with a functor, then all obstructions to its deformation are cocycles.
\end{theorem}

In principal, it appears, one could apply the polygonal method to directly show that obstructions to the deformation of any given (arbitrarily complicated) pasting diagram are cocycles -- however a metatheorem to this effect has proved to be beyond the authors' capabilities.  Instead, we will approach the general problem indirectly by reducing the deformations of any pasting diagram to deformations of a related pasting diagram all of whose cells are one of a small finite set of forms, provided one can specify that certain cells are deformed trivially:  the composition-free pasting diagrams for which the result was established in \cite{Yetter}, those of \ref{singlecomposition} and a short list of diagrams in which specified cells are deformed trivially:  triangles and solid tetrahedra all of whose faces are identity natural transformations and are deformed trivially, and two diagrams derived from those of Examples \ref{equalspost1comp} and \ref{equalspre1comp} by replacing the (degenerate) square labeled by $\tau$ with a bigon whose edged are labeled by the composite functors and two triangles labeled with the identity arrows of the composite functors, which must be deformed trivially.

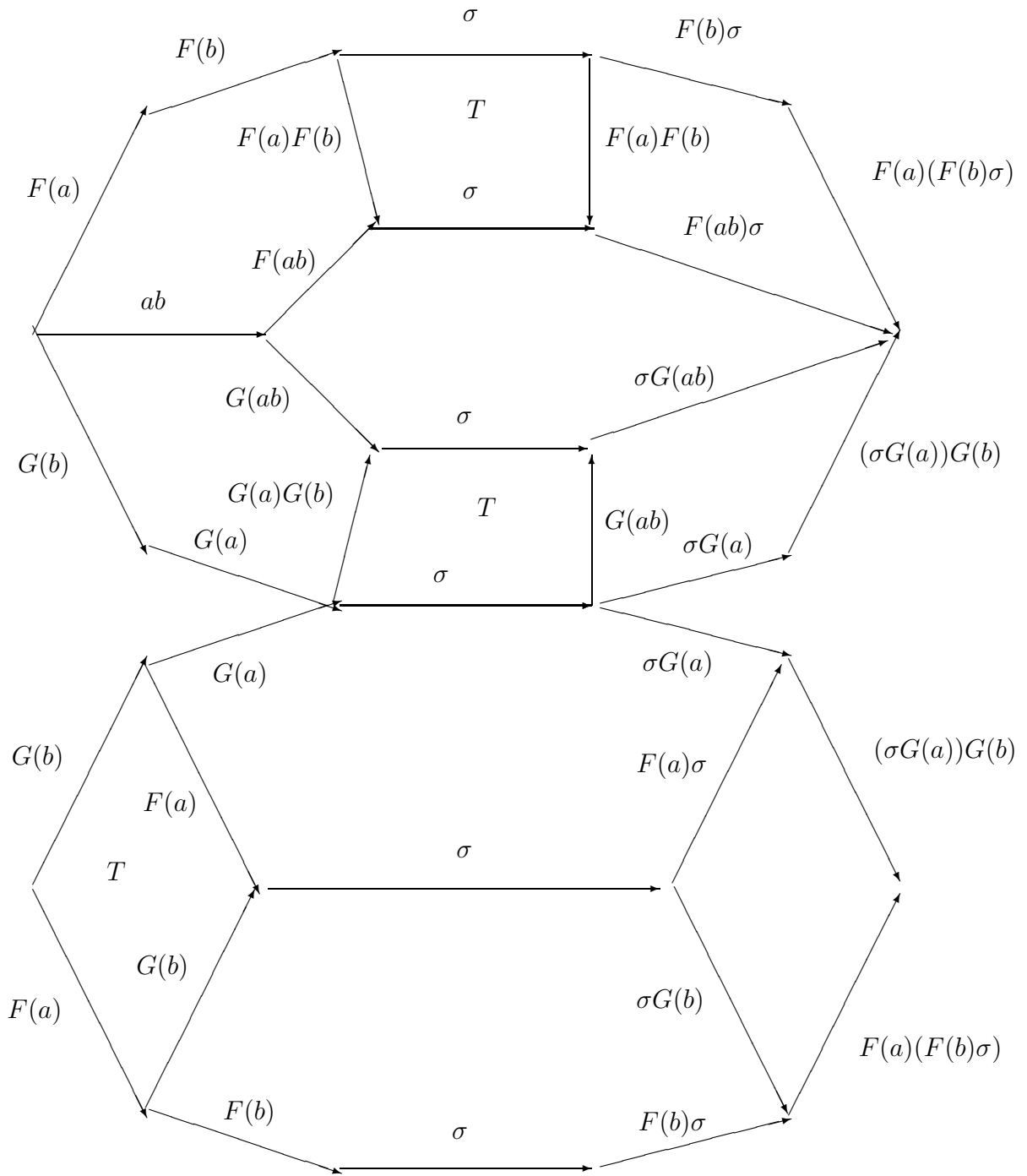
\begin{figure}
\begin{center}

\setlength{\unitlength}{4144sp}%
\begingroup\makeatletter\ifx\SetFigFontNFSS\undefined%
\gdef\SetFigFontNFSS#1#2#3#4#5{%
  \reset@font\fontsize{#1}{#2pt}%
  \fontfamily{#3}\fontseries{#4}\fontshape{#5}%
  \selectfont}%
\fi\endgroup%
\begin{picture}(6350,8297)(721,-8943)
\thinlines
{\color[rgb]{0,0,0}\put(916,-2926){\vector( 1,-2){810}}
}%
{\color[rgb]{0,0,0}\put(1741,-4486){\vector( 3,-1){1365}}
}%
{\color[rgb]{0,0,0}\put(3091,-4906){\vector( 1, 0){1785}}
}%
{\color[rgb]{0,0,0}\put(4936,-4891){\vector( 4, 1){1352}}
}%
{\color[rgb]{0,0,0}\put(6271,-4531){\vector( 1, 2){788}}
}%
{\color[rgb]{0,0,0}\put(916,-6886){\vector( 1, 2){810}}
}%
{\color[rgb]{0,0,0}\put(1741,-5326){\vector( 3, 1){1365}}
}%
{\color[rgb]{0,0,0}\put(3091,-4906){\vector( 1, 0){1785}}
}%
{\color[rgb]{0,0,0}\put(4936,-4921){\vector( 4,-1){1352}}
}%
{\color[rgb]{0,0,0}\put(6271,-5281){\vector( 1,-2){788}}
}%
{\color[rgb]{0,0,0}\put(916,-6916){\vector( 1,-2){810}}
}%
{\color[rgb]{0,0,0}\put(1741,-8476){\vector( 3,-1){1365}}
}%
{\color[rgb]{0,0,0}\put(3091,-8896){\vector( 1, 0){1785}}
}%
{\color[rgb]{0,0,0}\put(4936,-8881){\vector( 4, 1){1352}}
}%
{\color[rgb]{0,0,0}\put(6271,-8521){\vector( 1, 2){788}}
}%
{\color[rgb]{0,0,0}\put(916,-2986){\vector( 1, 2){810}}
}%
{\color[rgb]{0,0,0}\put(1741,-1426){\vector( 3, 1){1365}}
}%
{\color[rgb]{0,0,0}\put(3091,-1006){\vector( 1, 0){1785}}
}%
{\color[rgb]{0,0,0}\put(4936,-1021){\vector( 4,-1){1352}}
}%
{\color[rgb]{0,0,0}\put(6271,-1381){\vector( 1,-2){788}}
}%
{\color[rgb]{0,0,0}\put(1699,-5302){\vector( 1,-2){825}}
}%
{\color[rgb]{0,0,0}\put(1717,-8464){\vector( 1, 2){773}}
}%
{\color[rgb]{0,0,0}\put(5446,-6871){\vector( 1, 2){773}}
}%
{\color[rgb]{0,0,0}\put(5461,-6901){\vector( 1,-2){803}}
}%
{\color[rgb]{0,0,0}\put(2581,-6916){\vector( 1, 0){2775}}
}%
{\color[rgb]{0,0,0}\put(946,-2986){\vector( 1, 0){1620}}
}%
{\color[rgb]{0,0,0}\put(2551,-2986){\vector( 1, 1){795}}
}%
{\color[rgb]{0,0,0}\put(2573,-3023){\vector( 1,-1){795}}
}%
{\color[rgb]{0,0,0}\put(4906,-2281){\vector( 3,-1){2100}}
}%
{\color[rgb]{0,0,0}\put(4872,-3732){\vector( 3, 1){2100}}
}%
{\color[rgb]{0,0,0}\put(3301,-2236){\vector( 1, 0){1590}}
}%
{\color[rgb]{0,0,0}\put(3391,-3796){\vector( 1, 0){1455}}
}%
{\color[rgb]{0,0,0}\put(3076,-1036){\vector( 1,-4){293}}
}%
{\color[rgb]{0,0,0}\put(3046,-4876){\vector( 1, 4){259}}
}%
{\color[rgb]{0,0,0}\put(4861,-1036){\vector( 0,-1){1170}}
}%
{\color[rgb]{0,0,0}\put(4876,-4906){\vector( 0, 1){1065}}
}%
\put(871,-2011){\makebox(0,0)[lb]{\smash{{\SetFigFontNFSS{12}{14.4}{\rmdefault}{\mddefault}{\updefault}{\color[rgb]{0,0,0}$F(a)$}%
}}}}
\put(1921,-1021){\makebox(0,0)[lb]{\smash{{\SetFigFontNFSS{12}{14.4}{\rmdefault}{\mddefault}{\updefault}{\color[rgb]{0,0,0}$F(b)$}%
}}}}
\put(3961,-766){\makebox(0,0)[lb]{\smash{{\SetFigFontNFSS{12}{14.4}{\rmdefault}{\mddefault}{\updefault}{\color[rgb]{0,0,0}$\sigma$}%
}}}}
\put(5461,-886){\makebox(0,0)[lb]{\smash{{\SetFigFontNFSS{12}{14.4}{\rmdefault}{\mddefault}{\updefault}{\color[rgb]{0,0,0}$F(b)\sigma$}%
}}}}
\put(6856,-1891){\makebox(0,0)[lb]{\smash{{\SetFigFontNFSS{12}{14.4}{\rmdefault}{\mddefault}{\updefault}{\color[rgb]{0,0,0}$F(a)(F(b)\sigma)$}%
}}}}
\put(1681,-2806){\makebox(0,0)[lb]{\smash{{\SetFigFontNFSS{12}{14.4}{\rmdefault}{\mddefault}{\updefault}{\color[rgb]{0,0,0}$ab$}%
}}}}
\put(2356,-1651){\makebox(0,0)[lb]{\smash{{\SetFigFontNFSS{12}{14.4}{\rmdefault}{\mddefault}{\updefault}{\color[rgb]{0,0,0}$F(a)F(b)$}%
}}}}
\put(2461,-2521){\makebox(0,0)[lb]{\smash{{\SetFigFontNFSS{12}{14.4}{\rmdefault}{\mddefault}{\updefault}{\color[rgb]{0,0,0}$F(ab)$}%
}}}}
\put(3961,-2026){\makebox(0,0)[lb]{\smash{{\SetFigFontNFSS{12}{14.4}{\rmdefault}{\mddefault}{\updefault}{\color[rgb]{0,0,0}$\sigma$}%
}}}}
\put(4966,-1651){\makebox(0,0)[lb]{\smash{{\SetFigFontNFSS{12}{14.4}{\rmdefault}{\mddefault}{\updefault}{\color[rgb]{0,0,0}$F(a)F(b)$}%
}}}}
\put(5521,-2281){\makebox(0,0)[lb]{\smash{{\SetFigFontNFSS{12}{14.4}{\rmdefault}{\mddefault}{\updefault}{\color[rgb]{0,0,0}$F(ab)\sigma$}%
}}}}
\put(811,-3961){\makebox(0,0)[lb]{\smash{{\SetFigFontNFSS{12}{14.4}{\rmdefault}{\mddefault}{\updefault}{\color[rgb]{0,0,0}$G(b)$}%
}}}}
\put(2281,-3496){\makebox(0,0)[lb]{\smash{{\SetFigFontNFSS{12}{14.4}{\rmdefault}{\mddefault}{\updefault}{\color[rgb]{0,0,0}$G(ab)$}%
}}}}
\put(2056,-4501){\makebox(0,0)[lb]{\smash{{\SetFigFontNFSS{12}{14.4}{\rmdefault}{\mddefault}{\updefault}{\color[rgb]{0,0,0}$G(a)$}%
}}}}
\put(2296,-4171){\makebox(0,0)[lb]{\smash{{\SetFigFontNFSS{12}{14.4}{\rmdefault}{\mddefault}{\updefault}{\color[rgb]{0,0,0}$G(a)G(b)$}%
}}}}
\put(3916,-3616){\makebox(0,0)[lb]{\smash{{\SetFigFontNFSS{12}{14.4}{\rmdefault}{\mddefault}{\updefault}{\color[rgb]{0,0,0}$\sigma$}%
}}}}
\put(3751,-4741){\makebox(0,0)[lb]{\smash{{\SetFigFontNFSS{12}{14.4}{\rmdefault}{\mddefault}{\updefault}{\color[rgb]{0,0,0}$\sigma$}%
}}}}
\put(5176,-3346){\makebox(0,0)[lb]{\smash{{\SetFigFontNFSS{12}{14.4}{\rmdefault}{\mddefault}{\updefault}{\color[rgb]{0,0,0}$\sigma G(ab)$}%
}}}}
\put(3916,-6691){\makebox(0,0)[lb]{\smash{{\SetFigFontNFSS{12}{14.4}{\rmdefault}{\mddefault}{\updefault}{\color[rgb]{0,0,0}$\sigma$}%
}}}}
\put(3886,-8686){\makebox(0,0)[lb]{\smash{{\SetFigFontNFSS{12}{14.4}{\rmdefault}{\mddefault}{\updefault}{\color[rgb]{0,0,0}$\sigma$}%
}}}}
\put(4966,-4366){\makebox(0,0)[lb]{\smash{{\SetFigFontNFSS{12}{14.4}{\rmdefault}{\mddefault}{\updefault}{\color[rgb]{0,0,0}$G(ab)$}%
}}}}
\put(5521,-4531){\makebox(0,0)[lb]{\smash{{\SetFigFontNFSS{12}{14.4}{\rmdefault}{\mddefault}{\updefault}{\color[rgb]{0,0,0}$\sigma G(a)$}%
}}}}
\put(6766,-3886){\makebox(0,0)[lb]{\smash{{\SetFigFontNFSS{12}{14.4}{\rmdefault}{\mddefault}{\updefault}{\color[rgb]{0,0,0}$(\sigma G(a))G(b)$}%
}}}}
\put(766,-6031){\makebox(0,0)[lb]{\smash{{\SetFigFontNFSS{12}{14.4}{\rmdefault}{\mddefault}{\updefault}{\color[rgb]{0,0,0}$G(b)$}%
}}}}
\put(736,-7831){\makebox(0,0)[lb]{\smash{{\SetFigFontNFSS{12}{14.4}{\rmdefault}{\mddefault}{\updefault}{\color[rgb]{0,0,0}$F(a)$}%
}}}}
\put(6766,-8101){\makebox(0,0)[lb]{\smash{{\SetFigFontNFSS{12}{14.4}{\rmdefault}{\mddefault}{\updefault}{\color[rgb]{0,0,0}$F(a)(F(b)\sigma )$}%
}}}}
\put(6871,-5986){\makebox(0,0)[lb]{\smash{{\SetFigFontNFSS{12}{14.4}{\rmdefault}{\mddefault}{\updefault}{\color[rgb]{0,0,0}$(\sigma G(a))G(b)$}%
}}}}
\put(2191,-5446){\makebox(0,0)[lb]{\smash{{\SetFigFontNFSS{12}{14.4}{\rmdefault}{\mddefault}{\updefault}{\color[rgb]{0,0,0}$G(a)$}%
}}}}
\put(5236,-5371){\makebox(0,0)[lb]{\smash{{\SetFigFontNFSS{12}{14.4}{\rmdefault}{\mddefault}{\updefault}{\color[rgb]{0,0,0}$\sigma G(a)$}%
}}}}
\put(5206,-8626){\makebox(0,0)[lb]{\smash{{\SetFigFontNFSS{12}{14.4}{\rmdefault}{\mddefault}{\updefault}{\color[rgb]{0,0,0}$F(b)\sigma$}%
}}}}
\put(2266,-8551){\makebox(0,0)[lb]{\smash{{\SetFigFontNFSS{12}{14.4}{\rmdefault}{\mddefault}{\updefault}{\color[rgb]{0,0,0}$F(b)$}%
}}}}
\put(1651,-7516){\makebox(0,0)[lb]{\smash{{\SetFigFontNFSS{12}{14.4}{\rmdefault}{\mddefault}{\updefault}{\color[rgb]{0,0,0}$G(b)$}%
}}}}
\put(1696,-6361){\makebox(0,0)[lb]{\smash{{\SetFigFontNFSS{12}{14.4}{\rmdefault}{\mddefault}{\updefault}{\color[rgb]{0,0,0}$F(a)$}%
}}}}
\put(5191,-6091){\makebox(0,0)[lb]{\smash{{\SetFigFontNFSS{12}{14.4}{\rmdefault}{\mddefault}{\updefault}{\color[rgb]{0,0,0}$F(a)\sigma$}%
}}}}
\put(5191,-7756){\makebox(0,0)[lb]{\smash{{\SetFigFontNFSS{12}{14.4}{\rmdefault}{\mddefault}{\updefault}{\color[rgb]{0,0,0}$\sigma G(b)$}%
}}}}
\put(3991,-1456){\makebox(0,0)[lb]{\smash{{\SetFigFontNFSS{12}{14.4}{\rmdefault}{\mddefault}{\updefault}{\color[rgb]{0,0,0}$T$}%
}}}}
\put(4066,-4276){\makebox(0,0)[lb]{\smash{{\SetFigFontNFSS{12}{14.4}{\rmdefault}{\mddefault}{\updefault}{\color[rgb]{0,0,0}$T$}%
}}}}
\put(1441,-6856){\makebox(0,0)[lb]{\smash{{\SetFigFontNFSS{12}{14.4}{\rmdefault}{\mddefault}{\updefault}{\color[rgb]{0,0,0}$T$}%
}}}}
\end{picture}%

\end{center}
\caption{Obstructions to deforming naturality are cocycle \label{naturality_obstructions_are_cocycles}}
\end{figure}

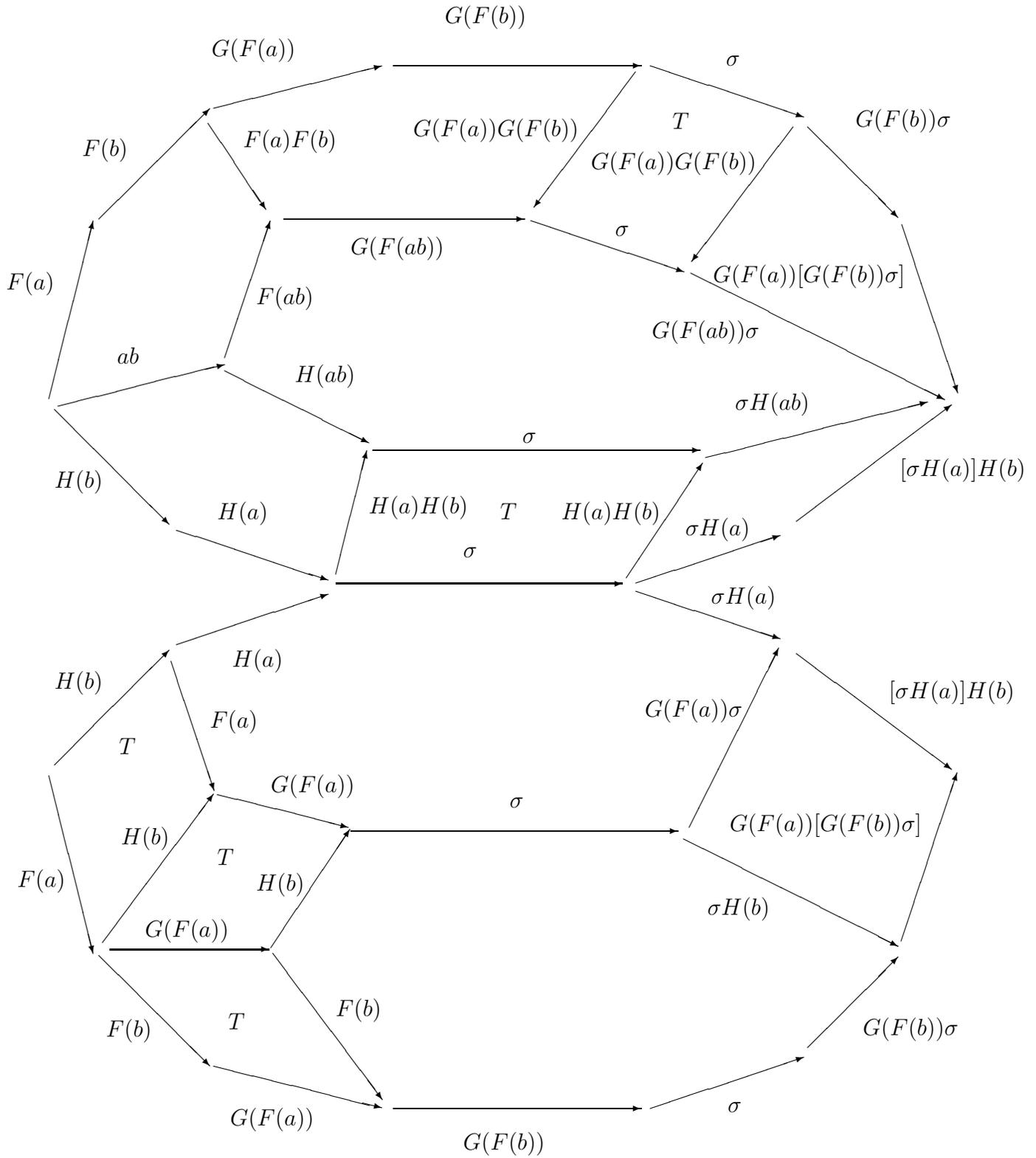
\begin{figure}
\begin{center}

\setlength{\unitlength}{4144sp}%
\begingroup\makeatletter\ifx\SetFigFontNFSS\undefined%
\gdef\SetFigFontNFSS#1#2#3#4#5{%
  \reset@font\fontsize{#1}{#2pt}%
  \fontfamily{#3}\fontseries{#4}\fontshape{#5}%
  \selectfont}%
\fi\endgroup%
\begin{picture}(7742,9379)(331,-10274)
\thinlines
{\color[rgb]{0,0,0}\put(691,-7171){\vector( 1,-4){360}}
}%
{\color[rgb]{0,0,0}\put(1096,-8626){\vector( 1,-1){900}}
}%
{\color[rgb]{0,0,0}\put(2026,-9526){\vector( 4,-1){1364}}
}%
{\color[rgb]{0,0,0}\put(3481,-9871){\vector( 1, 0){2010}}
}%
{\color[rgb]{0,0,0}\put(5566,-9871){\vector( 3, 1){1245}}
}%
{\color[rgb]{0,0,0}\put(6841,-9376){\vector( 1, 1){735}}
}%
{\color[rgb]{0,0,0}\put(736,-7081){\vector( 1, 1){930}}
}%
{\color[rgb]{0,0,0}\put(1726,-6106){\vector( 3, 1){1230}}
}%
{\color[rgb]{0,0,0}\put(5450,-5673){\vector( 3,-1){1170}}
}%
{\color[rgb]{0,0,0}\put(6750,-6177){\vector( 4,-3){1260}}
}%
{\color[rgb]{0,0,0}\put(691,-4111){\vector( 1, 4){360}}
}%
{\color[rgb]{0,0,0}\put(1096,-2656){\vector( 1, 1){900}}
}%
{\color[rgb]{0,0,0}\put(2026,-1756){\vector( 4, 1){1364}}
}%
{\color[rgb]{0,0,0}\put(3481,-1411){\vector( 1, 0){2010}}
}%
{\color[rgb]{0,0,0}\put(5566,-1411){\vector( 3,-1){1245}}
}%
{\color[rgb]{0,0,0}\put(6841,-1906){\vector( 1,-1){735}}
}%
{\color[rgb]{0,0,0}\put(7606,-2701){\vector( 1,-3){455}}
}%
{\color[rgb]{0,0,0}\put(736,-4201){\vector( 1,-1){930}}
}%
{\color[rgb]{0,0,0}\put(1726,-5176){\vector( 3,-1){1230}}
}%
{\color[rgb]{0,0,0}\put(3016,-5611){\vector( 1, 0){2325}}
}%
{\color[rgb]{0,0,0}\put(5450,-5609){\vector( 3, 1){1170}}
}%
{\color[rgb]{0,0,0}\put(6750,-5105){\vector( 4, 3){1260}}
}%
{\color[rgb]{0,0,0}\put(766,-4171){\vector( 4, 1){1352}}
}%
{\color[rgb]{0,0,0}\put(2116,-3781){\vector( 1, 3){370}}
}%
{\color[rgb]{0,0,0}\put(1981,-1876){\vector( 2,-3){470}}
}%
{\color[rgb]{0,0,0}\put(2596,-2656){\vector( 1, 0){1950}}
}%
{\color[rgb]{0,0,0}\put(4597,-2668){\vector( 3,-1){1245}}
}%
{\color[rgb]{0,0,0}\put(5896,-3091){\vector( 2,-1){2056}}
}%
{\color[rgb]{0,0,0}\put(2116,-3886){\vector( 2,-1){1170}}
}%
{\color[rgb]{0,0,0}\put(3016,-5536){\vector( 1, 4){251}}
}%
{\color[rgb]{0,0,0}\put(6016,-4591){\vector( 4, 1){1800}}
}%
{\color[rgb]{0,0,0}\put(3316,-4531){\vector( 1, 0){2655}}
}%
{\color[rgb]{0,0,0}\put(5371,-5566){\vector( 2, 3){620}}
}%
{\color[rgb]{0,0,0}\put(1126,-8521){\vector( 3, 4){900}}
}%
{\color[rgb]{0,0,0}\put(1681,-6241){\vector( 1,-3){350}}
}%
{\color[rgb]{0,0,0}\put(1186,-8581){\vector( 1, 0){1290}}
}%
{\color[rgb]{0,0,0}\put(2506,-8611){\vector( 3,-4){888}}
}%
{\color[rgb]{0,0,0}\put(2056,-7321){\vector( 4,-1){1064}}
}%
{\color[rgb]{0,0,0}\put(2491,-8566){\vector( 2, 3){640}}
}%
{\color[rgb]{0,0,0}\put(3136,-7621){\vector( 1, 0){2655}}
}%
{\color[rgb]{0,0,0}\put(5881,-7591){\vector( 1, 2){728}}
}%
{\color[rgb]{0,0,0}\put(5836,-7681){\vector( 2,-1){1740}}
}%
{\color[rgb]{0,0,0}\put(7600,-8509){\vector( 1, 3){455}}
}%
{\color[rgb]{0,0,0}\put(5447,-1464){\vector(-3,-4){822}}
}%
{\color[rgb]{0,0,0}\put(6725,-1905){\vector(-3,-4){822}}
}%
\put(4051,-5401){\makebox(0,0)[lb]{\smash{{\SetFigFontNFSS{12}{14.4}{\rmdefault}{\mddefault}{\updefault}{\color[rgb]{0,0,0}$\sigma$}%
}}}}
\put(6181,-1411){\makebox(0,0)[lb]{\smash{{\SetFigFontNFSS{12}{14.4}{\rmdefault}{\mddefault}{\updefault}{\color[rgb]{0,0,0}$\sigma$}%
}}}}
\put(6196,-9901){\makebox(0,0)[lb]{\smash{{\SetFigFontNFSS{12}{14.4}{\rmdefault}{\mddefault}{\updefault}{\color[rgb]{0,0,0}$\sigma$}%
}}}}
\put(4531,-4471){\makebox(0,0)[lb]{\smash{{\SetFigFontNFSS{12}{14.4}{\rmdefault}{\mddefault}{\updefault}{\color[rgb]{0,0,0}$\sigma$}%
}}}}
\put(5281,-2791){\makebox(0,0)[lb]{\smash{{\SetFigFontNFSS{12}{14.4}{\rmdefault}{\mddefault}{\updefault}{\color[rgb]{0,0,0}$\sigma$}%
}}}}
\put(346,-3226){\makebox(0,0)[lb]{\smash{{\SetFigFontNFSS{12}{14.4}{\rmdefault}{\mddefault}{\updefault}{\color[rgb]{0,0,0}$F(a)$}%
}}}}
\put(961,-2146){\makebox(0,0)[lb]{\smash{{\SetFigFontNFSS{12}{14.4}{\rmdefault}{\mddefault}{\updefault}{\color[rgb]{0,0,0}$F(b)$}%
}}}}
\put(1246,-3826){\makebox(0,0)[lb]{\smash{{\SetFigFontNFSS{12}{14.4}{\rmdefault}{\mddefault}{\updefault}{\color[rgb]{0,0,0}$ab$}%
}}}}
\put(2011,-1336){\makebox(0,0)[lb]{\smash{{\SetFigFontNFSS{12}{14.4}{\rmdefault}{\mddefault}{\updefault}{\color[rgb]{0,0,0}$G(F(a))$}%
}}}}
\put(3901,-1066){\makebox(0,0)[lb]{\smash{{\SetFigFontNFSS{12}{14.4}{\rmdefault}{\mddefault}{\updefault}{\color[rgb]{0,0,0}$G(F(b))$}%
}}}}
\put(2371,-3346){\makebox(0,0)[lb]{\smash{{\SetFigFontNFSS{12}{14.4}{\rmdefault}{\mddefault}{\updefault}{\color[rgb]{0,0,0}$F(ab)$}%
}}}}
\put(2266,-2071){\makebox(0,0)[lb]{\smash{{\SetFigFontNFSS{12}{14.4}{\rmdefault}{\mddefault}{\updefault}{\color[rgb]{0,0,0}$F(a)F(b)$}%
}}}}
\put(3136,-2941){\makebox(0,0)[lb]{\smash{{\SetFigFontNFSS{12}{14.4}{\rmdefault}{\mddefault}{\updefault}{\color[rgb]{0,0,0}$G(F(ab))$}%
}}}}
\put(3646,-1981){\makebox(0,0)[lb]{\smash{{\SetFigFontNFSS{12}{14.4}{\rmdefault}{\mddefault}{\updefault}{\color[rgb]{0,0,0}$G(F(a))G(F(b))$}%
}}}}
\put(5086,-2251){\makebox(0,0)[lb]{\smash{{\SetFigFontNFSS{12}{14.4}{\rmdefault}{\mddefault}{\updefault}{\color[rgb]{0,0,0}$G(F(a))G(F(b))$}%
}}}}
\put(5581,-3616){\makebox(0,0)[lb]{\smash{{\SetFigFontNFSS{12}{14.4}{\rmdefault}{\mddefault}{\updefault}{\color[rgb]{0,0,0}$G(F(ab))\sigma $}%
}}}}
\put(7231,-1906){\makebox(0,0)[lb]{\smash{{\SetFigFontNFSS{12}{14.4}{\rmdefault}{\mddefault}{\updefault}{\color[rgb]{0,0,0}$G(F(b))\sigma $}%
}}}}
\put(6076,-3166){\makebox(0,0)[lb]{\smash{{\SetFigFontNFSS{12}{14.4}{\rmdefault}{\mddefault}{\updefault}{\color[rgb]{0,0,0}$G(F(a))[G(F(b))\sigma ]$}%
}}}}
\put(7291,-9271){\makebox(0,0)[lb]{\smash{{\SetFigFontNFSS{12}{14.4}{\rmdefault}{\mddefault}{\updefault}{\color[rgb]{0,0,0}$G(F(b))\sigma$}%
}}}}
\put(6211,-7621){\makebox(0,0)[lb]{\smash{{\SetFigFontNFSS{12}{14.4}{\rmdefault}{\mddefault}{\updefault}{\color[rgb]{0,0,0}$G(F(a))[G(F(b))\sigma]$}%
}}}}
\put(436,-8071){\makebox(0,0)[lb]{\smash{{\SetFigFontNFSS{12}{14.4}{\rmdefault}{\mddefault}{\updefault}{\color[rgb]{0,0,0}$F(a)$}%
}}}}
\put(1156,-9286){\makebox(0,0)[lb]{\smash{{\SetFigFontNFSS{12}{14.4}{\rmdefault}{\mddefault}{\updefault}{\color[rgb]{0,0,0}$F(b)$}%
}}}}
\put(2161,-10006){\makebox(0,0)[lb]{\smash{{\SetFigFontNFSS{12}{14.4}{\rmdefault}{\mddefault}{\updefault}{\color[rgb]{0,0,0}$G(F(a))$}%
}}}}
\put(4051,-10201){\makebox(0,0)[lb]{\smash{{\SetFigFontNFSS{12}{14.4}{\rmdefault}{\mddefault}{\updefault}{\color[rgb]{0,0,0}$G(F(b))$}%
}}}}
\put(736,-4831){\makebox(0,0)[lb]{\smash{{\SetFigFontNFSS{12}{14.4}{\rmdefault}{\mddefault}{\updefault}{\color[rgb]{0,0,0}$H(b)$}%
}}}}
\put(2056,-5086){\makebox(0,0)[lb]{\smash{{\SetFigFontNFSS{12}{14.4}{\rmdefault}{\mddefault}{\updefault}{\color[rgb]{0,0,0}$H(a)$}%
}}}}
\put(3286,-5041){\makebox(0,0)[lb]{\smash{{\SetFigFontNFSS{12}{14.4}{\rmdefault}{\mddefault}{\updefault}{\color[rgb]{0,0,0}$H(a)H(b)$}%
}}}}
\put(2686,-3976){\makebox(0,0)[lb]{\smash{{\SetFigFontNFSS{12}{14.4}{\rmdefault}{\mddefault}{\updefault}{\color[rgb]{0,0,0}$H(ab)$}%
}}}}
\put(736,-6466){\makebox(0,0)[lb]{\smash{{\SetFigFontNFSS{12}{14.4}{\rmdefault}{\mddefault}{\updefault}{\color[rgb]{0,0,0}$H(b)$}%
}}}}
\put(2176,-6286){\makebox(0,0)[lb]{\smash{{\SetFigFontNFSS{12}{14.4}{\rmdefault}{\mddefault}{\updefault}{\color[rgb]{0,0,0}$H(a)$}%
}}}}
\put(4846,-5086){\makebox(0,0)[lb]{\smash{{\SetFigFontNFSS{12}{14.4}{\rmdefault}{\mddefault}{\updefault}{\color[rgb]{0,0,0}$H(a)H(b)$}%
}}}}
\put(6256,-4201){\makebox(0,0)[lb]{\smash{{\SetFigFontNFSS{12}{14.4}{\rmdefault}{\mddefault}{\updefault}{\color[rgb]{0,0,0}$\sigma H(ab)$}%
}}}}
\put(5851,-5221){\makebox(0,0)[lb]{\smash{{\SetFigFontNFSS{12}{14.4}{\rmdefault}{\mddefault}{\updefault}{\color[rgb]{0,0,0}$\sigma H(a)$}%
}}}}
\put(7591,-4741){\makebox(0,0)[lb]{\smash{{\SetFigFontNFSS{12}{14.4}{\rmdefault}{\mddefault}{\updefault}{\color[rgb]{0,0,0}$[\sigma H(a)]H(b)$}%
}}}}
\put(6061,-5776){\makebox(0,0)[lb]{\smash{{\SetFigFontNFSS{12}{14.4}{\rmdefault}{\mddefault}{\updefault}{\color[rgb]{0,0,0}$\sigma H(a)$}%
}}}}
\put(7501,-6541){\makebox(0,0)[lb]{\smash{{\SetFigFontNFSS{12}{14.4}{\rmdefault}{\mddefault}{\updefault}{\color[rgb]{0,0,0}$[\sigma H(a)]H(b)$}%
}}}}
\put(4426,-7426){\makebox(0,0)[lb]{\smash{{\SetFigFontNFSS{12}{14.4}{\rmdefault}{\mddefault}{\updefault}{\color[rgb]{0,0,0}$\sigma$}%
}}}}
\put(1996,-6781){\makebox(0,0)[lb]{\smash{{\SetFigFontNFSS{12}{14.4}{\rmdefault}{\mddefault}{\updefault}{\color[rgb]{0,0,0}$F(a)$}%
}}}}
\put(1276,-7726){\makebox(0,0)[lb]{\smash{{\SetFigFontNFSS{12}{14.4}{\rmdefault}{\mddefault}{\updefault}{\color[rgb]{0,0,0}$H(b)$}%
}}}}
\put(2491,-7306){\makebox(0,0)[lb]{\smash{{\SetFigFontNFSS{12}{14.4}{\rmdefault}{\mddefault}{\updefault}{\color[rgb]{0,0,0}$G(F(a))$}%
}}}}
\put(1471,-8476){\makebox(0,0)[lb]{\smash{{\SetFigFontNFSS{12}{14.4}{\rmdefault}{\mddefault}{\updefault}{\color[rgb]{0,0,0}$G(F(a))$}%
}}}}
\put(2371,-8101){\makebox(0,0)[lb]{\smash{{\SetFigFontNFSS{12}{14.4}{\rmdefault}{\mddefault}{\updefault}{\color[rgb]{0,0,0}$H(b)$}%
}}}}
\put(3016,-9121){\makebox(0,0)[lb]{\smash{{\SetFigFontNFSS{12}{14.4}{\rmdefault}{\mddefault}{\updefault}{\color[rgb]{0,0,0}$F(b)$}%
}}}}
\put(5521,-6691){\makebox(0,0)[lb]{\smash{{\SetFigFontNFSS{12}{14.4}{\rmdefault}{\mddefault}{\updefault}{\color[rgb]{0,0,0}$G(F(a))\sigma$}%
}}}}
\put(6031,-8296){\makebox(0,0)[lb]{\smash{{\SetFigFontNFSS{12}{14.4}{\rmdefault}{\mddefault}{\updefault}{\color[rgb]{0,0,0}$\sigma H(b)$}%
}}}}
\put(1261,-6991){\makebox(0,0)[lb]{\smash{{\SetFigFontNFSS{12}{14.4}{\rmdefault}{\mddefault}{\updefault}{\color[rgb]{0,0,0}$T$}%
}}}}
\put(2056,-7891){\makebox(0,0)[lb]{\smash{{\SetFigFontNFSS{12}{14.4}{\rmdefault}{\mddefault}{\updefault}{\color[rgb]{0,0,0}$T$}%
}}}}
\put(2146,-9226){\makebox(0,0)[lb]{\smash{{\SetFigFontNFSS{12}{14.4}{\rmdefault}{\mddefault}{\updefault}{\color[rgb]{0,0,0}$T$}%
}}}}
\put(4351,-5086){\makebox(0,0)[lb]{\smash{{\SetFigFontNFSS{12}{14.4}{\rmdefault}{\mddefault}{\updefault}{\color[rgb]{0,0,0}$T$}%
}}}}
\put(5746,-1921){\makebox(0,0)[lb]{\smash{{\SetFigFontNFSS{12}{14.4}{\rmdefault}{\mddefault}{\updefault}{\color[rgb]{0,0,0}$T$}%
}}}}
\end{picture}%

\end{center}
\caption{Obstructions to deforming a natural transformation from a composition of functors to a functor are cocycle \label{nat_from_composition_obstructions_are_cocycles}}
\end{figure}

\begin{figure}
\begin{center}

\setlength{\unitlength}{4144sp}%
\begingroup\makeatletter\ifx\SetFigFontNFSS\undefined%
\gdef\SetFigFontNFSS#1#2#3#4#5{%
  \reset@font\fontsize{#1}{#2pt}%
  \fontfamily{#3}\fontseries{#4}\fontshape{#5}%
  \selectfont}%
\fi\endgroup%
\begin{picture}(7968,9132)(1309,-10003)
\thinlines
{\color[rgb]{0,0,0}\put(3286,-1216){\vector( 1,-1){1095}}
}%
{\color[rgb]{0,0,0}\put(5963,-2333){\vector( 1, 1){1095}}
}%
{\color[rgb]{0,0,0}\put(4366,-2341){\vector( 1, 0){1485}}
}%
{\color[rgb]{0,0,0}\put(7051,-4726){\vector( 4, 1){2116}}
}%
{\color[rgb]{0,0,0}\put(6151,-4171){\vector( 3,-2){840}}
}%
{\color[rgb]{0,0,0}\put(3496,-4981){\vector( 3, 1){1242.900}}
}%
{\color[rgb]{0,0,0}\put(4739,-4567){\vector( 3, 1){1322.100}}
}%
{\color[rgb]{0,0,0}\put(1430,-4293){\vector( 3,-1){2085}}
}%
{\color[rgb]{0,0,0}\put(7051,-3526){\vector( 4,-1){2116}}
}%
{\color[rgb]{0,0,0}\put(6151,-4081){\vector( 3, 2){840}}
}%
{\color[rgb]{0,0,0}\put(3496,-3271){\vector( 3,-1){1242.900}}
}%
{\color[rgb]{0,0,0}\put(4739,-3685){\vector( 3,-1){1322.100}}
}%
{\color[rgb]{0,0,0}\put(1430,-3959){\vector( 3, 1){2085}}
}%
{\color[rgb]{0,0,0}\put(3264,-7066){\vector( 1, 1){1095}}
}%
{\color[rgb]{0,0,0}\put(5941,-5949){\vector( 1,-1){1095}}
}%
{\color[rgb]{0,0,0}\put(4344,-5941){\vector( 1, 0){1485}}
}%
{\color[rgb]{0,0,0}\put(1321,-4111){\vector( 2, 3){1970}}
}%
{\color[rgb]{0,0,0}\put(3361,-1201){\vector( 1, 0){3750}}
}%
{\color[rgb]{0,0,0}\put(7126,-1216){\vector( 3,-4){2139}}
}%
{\color[rgb]{0,0,0}\put(1332,-4135){\vector( 2,-3){1970}}
}%
{\color[rgb]{0,0,0}\put(7100,-6996){\vector( 3, 4){2139}}
}%
{\color[rgb]{0,0,0}\put(3331,-7081){\vector( 1, 0){3750}}
}%
{\color[rgb]{0,0,0}\put(5911,-2386){\vector( 1,-1){1140}}
}%
{\color[rgb]{0,0,0}\put(3451,-3196){\vector( 1, 1){855}}
}%
{\color[rgb]{0,0,0}\put(3474,-5019){\vector( 1,-1){855}}
}%
{\color[rgb]{0,0,0}\put(5850,-5881){\vector( 1, 1){1140}}
}%
{\color[rgb]{0,0,0}\put(3286,-9961){\vector( 1, 0){3750}}
}%
{\color[rgb]{0,0,0}\put(2327,-8497){\vector( 2, 3){900}}
}%
{\color[rgb]{0,0,0}\put(2372,-8590){\vector( 2,-3){900}}
}%
{\color[rgb]{0,0,0}\put(7066,-7126){\vector( 2,-3){900}}
}%
{\color[rgb]{0,0,0}\put(7021,-9871){\vector( 2, 3){900}}
}%
\put(3046,-2776){\makebox(0,0)[lb]{\smash{{\SetFigFontNFSS{12}{14.4}{\rmdefault}{\mddefault}{\updefault}{\color[rgb]{0,0,0}$T$}%
}}}}
\put(3406,-4156){\makebox(0,0)[lb]{\smash{{\SetFigFontNFSS{12}{14.4}{\rmdefault}{\mddefault}{\updefault}{\color[rgb]{0,0,0}$T$}%
}}}}
\put(3031,-5716){\makebox(0,0)[lb]{\smash{{\SetFigFontNFSS{12}{14.4}{\rmdefault}{\mddefault}{\updefault}{\color[rgb]{0,0,0}$T$}%
}}}}
\put(5071,-1681){\makebox(0,0)[lb]{\smash{{\SetFigFontNFSS{12}{14.4}{\rmdefault}{\mddefault}{\updefault}{\color[rgb]{0,0,0}$*$}%
}}}}
\put(7261,-4096){\makebox(0,0)[lb]{\smash{{\SetFigFontNFSS{12}{14.4}{\rmdefault}{\mddefault}{\updefault}{\color[rgb]{0,0,0}$*$}%
}}}}
\put(1876,-5881){\makebox(0,0)[lb]{\smash{{\SetFigFontNFSS{12}{14.4}{\rmdefault}{\mddefault}{\updefault}{\color[rgb]{0,0,0}$K(a)$}%
}}}}
\put(1906,-2221){\makebox(0,0)[lb]{\smash{{\SetFigFontNFSS{12}{14.4}{\rmdefault}{\mddefault}{\updefault}{\color[rgb]{0,0,0}$F(a)$}%
}}}}
\put(3661,-3646){\makebox(0,0)[lb]{\smash{{\SetFigFontNFSS{12}{14.4}{\rmdefault}{\mddefault}{\updefault}{\color[rgb]{0,0,0}$G(a)$}%
}}}}
\put(3751,-4651){\makebox(0,0)[lb]{\smash{{\SetFigFontNFSS{12}{14.4}{\rmdefault}{\mddefault}{\updefault}{\color[rgb]{0,0,0}$G(a)$}%
}}}}
\put(2356,-4441){\makebox(0,0)[lb]{\smash{{\SetFigFontNFSS{12}{14.4}{\rmdefault}{\mddefault}{\updefault}{\color[rgb]{0,0,0}$\sigma$}%
}}}}
\put(3541,-1891){\makebox(0,0)[lb]{\smash{{\SetFigFontNFSS{12}{14.4}{\rmdefault}{\mddefault}{\updefault}{\color[rgb]{0,0,0}$\tau$}%
}}}}
\put(2356,-3436){\makebox(0,0)[lb]{\smash{{\SetFigFontNFSS{12}{14.4}{\rmdefault}{\mddefault}{\updefault}{\color[rgb]{0,0,0}$\tau$}%
}}}}
\put(5266,-4711){\makebox(0,0)[lb]{\smash{{\SetFigFontNFSS{12}{14.4}{\rmdefault}{\mddefault}{\updefault}{\color[rgb]{0,0,0}$\tau$}%
}}}}
\put(5131,-5716){\makebox(0,0)[lb]{\smash{{\SetFigFontNFSS{12}{14.4}{\rmdefault}{\mddefault}{\updefault}{\color[rgb]{0,0,0}$\tau$}%
}}}}
\put(4951,-2236){\makebox(0,0)[lb]{\smash{{\SetFigFontNFSS{12}{14.4}{\rmdefault}{\mddefault}{\updefault}{\color[rgb]{0,0,0}$\sigma$}%
}}}}
\put(4951,-991){\makebox(0,0)[lb]{\smash{{\SetFigFontNFSS{12}{14.4}{\rmdefault}{\mddefault}{\updefault}{\color[rgb]{0,0,0}$\upsilon$}%
}}}}
\put(5056,-6886){\makebox(0,0)[lb]{\smash{{\SetFigFontNFSS{12}{14.4}{\rmdefault}{\mddefault}{\updefault}{\color[rgb]{0,0,0}$\upsilon$}%
}}}}
\put(3916,-6586){\makebox(0,0)[lb]{\smash{{\SetFigFontNFSS{12}{14.4}{\rmdefault}{\mddefault}{\updefault}{\color[rgb]{0,0,0}$s$}%
}}}}
\put(6496,-6271){\makebox(0,0)[lb]{\smash{{\SetFigFontNFSS{12}{14.4}{\rmdefault}{\mddefault}{\updefault}{\color[rgb]{0,0,0}$\sigma \tau$}%
}}}}
\put(6541,-1951){\makebox(0,0)[lb]{\smash{{\SetFigFontNFSS{12}{14.4}{\rmdefault}{\mddefault}{\updefault}{\color[rgb]{0,0,0}$\sigma \tau$}%
}}}}
\put(8116,-2146){\makebox(0,0)[lb]{\smash{{\SetFigFontNFSS{12}{14.4}{\rmdefault}{\mddefault}{\updefault}{\color[rgb]{0,0,0}$F(a)\upsilon = F(a)\sigma \tau$}%
}}}}
\put(8326,-5881){\makebox(0,0)[lb]{\smash{{\SetFigFontNFSS{12}{14.4}{\rmdefault}{\mddefault}{\updefault}{\color[rgb]{0,0,0}$\upsilon K(a) = \sigma \tau K(a)$}%
}}}}
\put(6511,-5461){\makebox(0,0)[lb]{\smash{{\SetFigFontNFSS{12}{14.4}{\rmdefault}{\mddefault}{\updefault}{\color[rgb]{0,0,0}$\tau K(a)$}%
}}}}
\put(4006,-5356){\makebox(0,0)[lb]{\smash{{\SetFigFontNFSS{12}{14.4}{\rmdefault}{\mddefault}{\updefault}{\color[rgb]{0,0,0}$K(a)$}%
}}}}
\put(4006,-2881){\makebox(0,0)[lb]{\smash{{\SetFigFontNFSS{12}{14.4}{\rmdefault}{\mddefault}{\updefault}{\color[rgb]{0,0,0}$F(a)$}%
}}}}
\put(5206,-3661){\makebox(0,0)[lb]{\smash{{\SetFigFontNFSS{12}{14.4}{\rmdefault}{\mddefault}{\updefault}{\color[rgb]{0,0,0}$\sigma$}%
}}}}
\put(6106,-3676){\makebox(0,0)[lb]{\smash{{\SetFigFontNFSS{12}{14.4}{\rmdefault}{\mddefault}{\updefault}{\color[rgb]{0,0,0}$\sigma G(a)$}%
}}}}
\put(6001,-4621){\makebox(0,0)[lb]{\smash{{\SetFigFontNFSS{12}{14.4}{\rmdefault}{\mddefault}{\updefault}{\color[rgb]{0,0,0}$G(a)\tau$}%
}}}}
\put(7576,-3421){\makebox(0,0)[lb]{\smash{{\SetFigFontNFSS{12}{14.4}{\rmdefault}{\mddefault}{\updefault}{\color[rgb]{0,0,0}$(\sigma G(a))\tau$}%
}}}}
\put(7186,-4396){\makebox(0,0)[lb]{\smash{{\SetFigFontNFSS{12}{14.4}{\rmdefault}{\mddefault}{\updefault}{\color[rgb]{0,0,0}$\sigma (G(a)\tau )$}%
}}}}
\put(6436,-2656){\makebox(0,0)[lb]{\smash{{\SetFigFontNFSS{12}{14.4}{\rmdefault}{\mddefault}{\updefault}{\color[rgb]{0,0,0}$F(a)\sigma $}%
}}}}
\put(4951,-9781){\makebox(0,0)[lb]{\smash{{\SetFigFontNFSS{12}{14.4}{\rmdefault}{\mddefault}{\updefault}{\color[rgb]{0,0,0}$\upsilon$}%
}}}}
\put(2611,-7306){\makebox(0,0)[lb]{\smash{{\SetFigFontNFSS{12}{14.4}{\rmdefault}{\mddefault}{\updefault}{\color[rgb]{0,0,0}$K(a)$}%
}}}}
\put(2251,-9376){\makebox(0,0)[lb]{\smash{{\SetFigFontNFSS{12}{14.4}{\rmdefault}{\mddefault}{\updefault}{\color[rgb]{0,0,0}$F(a)$}%
}}}}
\put(7651,-7621){\makebox(0,0)[lb]{\smash{{\SetFigFontNFSS{12}{14.4}{\rmdefault}{\mddefault}{\updefault}{\color[rgb]{0,0,0}$\upsilon K(a) = \sigma \tau K(a)$}%
}}}}
\put(7606,-9421){\makebox(0,0)[lb]{\smash{{\SetFigFontNFSS{12}{14.4}{\rmdefault}{\mddefault}{\updefault}{\color[rgb]{0,0,0}$F(a)\upsilon = F(a)\sigma \tau$}%
}}}}
\end{picture}

\end{center}
\caption{Obstructions to deforming the equality between a 2-composition of natural transformations and a natural transformation are cocycles \label{2composition_obstructions_are_cocycles}}
\end{figure}
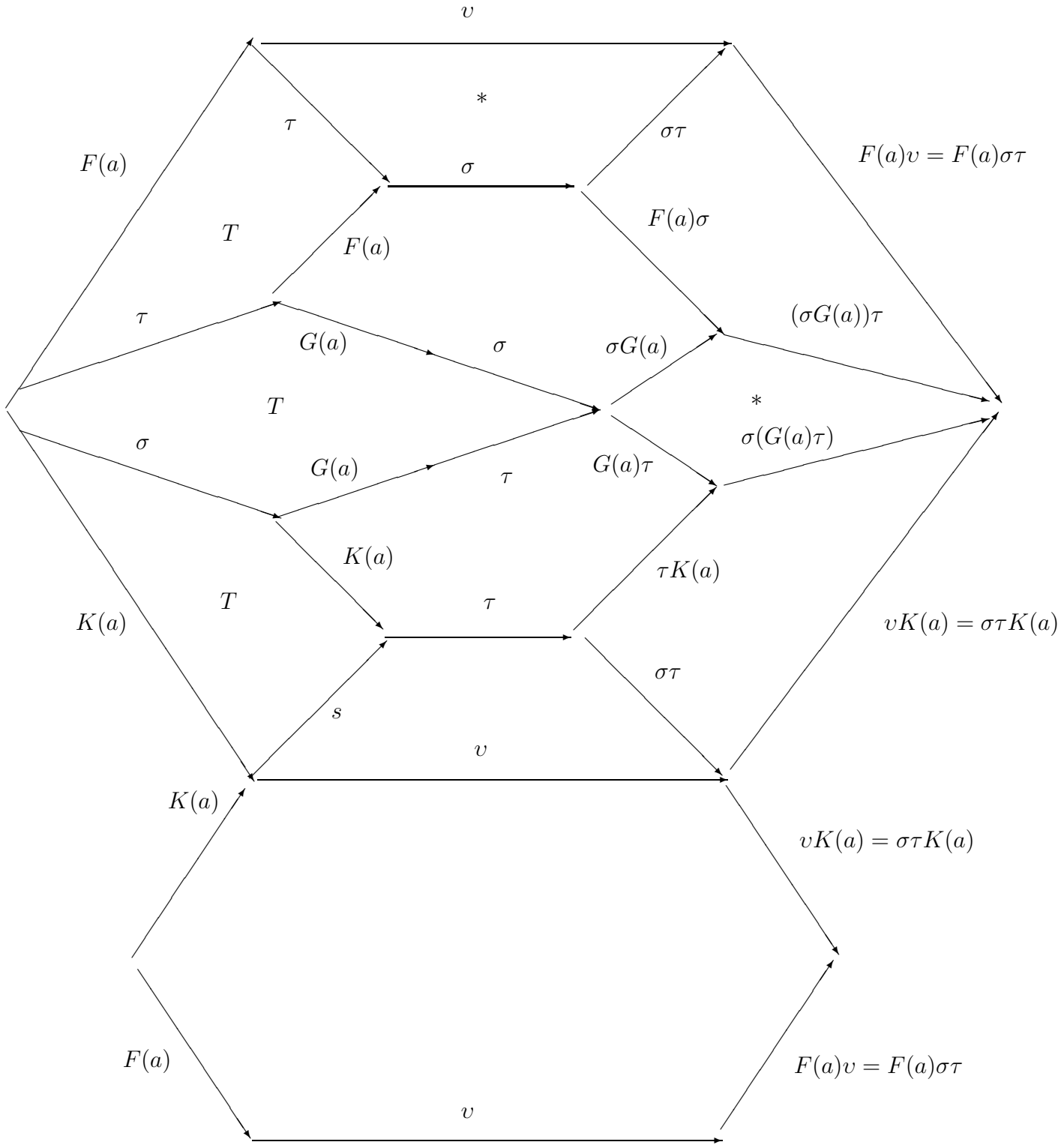

\begin{figure}
\begin{center}

\setlength{\unitlength}{4144sp}%
\begingroup\makeatletter\ifx\SetFigFontNFSS\undefined%
\gdef\SetFigFontNFSS#1#2#3#4#5{%
  \reset@font\fontsize{#1}{#2pt}%
  \fontfamily{#3}\fontseries{#4}\fontshape{#5}%
  \selectfont}%
\fi\endgroup%
\begin{picture}(6519,9027)(961,-9493)
\thinlines
{\color[rgb]{0,0,0}\put(1351,-7321){\vector( 1,-4){341}}
}%
{\color[rgb]{0,0,0}\put(1726,-8701){\vector( 2,-1){1516}}
}%
{\color[rgb]{0,0,0}\put(3316,-9451){\vector( 1, 0){2610}}
}%
{\color[rgb]{0,0,0}\put(5896,-9421){\vector( 3, 4){1542}}
}%
{\color[rgb]{0,0,0}\put(3256,-841){\vector( 3,-2){1335}}
}%
{\color[rgb]{0,0,0}\put(4636,-1756){\vector( 4, 3){1200}}
}%
{\color[rgb]{0,0,0}\put(3316,-5131){\vector( 3, 2){1335}}
}%
{\color[rgb]{0,0,0}\put(4696,-4216){\vector( 4,-3){1200}}
}%
{\color[rgb]{0,0,0}\put(1756,-1651){\vector( 4,-1){2144}}
}%
{\color[rgb]{0,0,0}\put(3901,-2221){\vector( 3, 2){645}}
}%
{\color[rgb]{0,0,0}\put(3976,-2296){\vector( 3,-2){1005}}
}%
{\color[rgb]{0,0,0}\put(1351,-2941){\vector( 1, 4){341}}
}%
{\color[rgb]{0,0,0}\put(1726,-1561){\vector( 2, 1){1516}}
}%
{\color[rgb]{0,0,0}\put(3316,-811){\vector( 1, 0){2610}}
}%
{\color[rgb]{0,0,0}\put(5896,-841){\vector( 3,-4){1542}}
}%
{\color[rgb]{0,0,0}\put(1756,-4281){\vector( 4, 1){2144}}
}%
{\color[rgb]{0,0,0}\put(3901,-3711){\vector( 3,-2){645}}
}%
{\color[rgb]{0,0,0}\put(3976,-3636){\vector( 3, 2){1005}}
}%
{\color[rgb]{0,0,0}\put(1366,-7321){\vector( 1, 4){341}}
}%
{\color[rgb]{0,0,0}\put(1741,-5941){\vector( 2, 1){1516}}
}%
{\color[rgb]{0,0,0}\put(5911,-5221){\vector( 3,-4){1542}}
}%
{\color[rgb]{0,0,0}\put(1381,-3031){\vector( 1,-4){341}}
}%
{\color[rgb]{0,0,0}\put(1756,-4411){\vector( 2,-1){1516}}
}%
{\color[rgb]{0,0,0}\put(3346,-5161){\vector( 1, 0){2610}}
}%
{\color[rgb]{0,0,0}\put(5926,-5131){\vector( 3, 4){1542}}
}%
{\color[rgb]{0,0,0}\put(4981,-3016){\vector( 1, 0){2400}}
}%
\put(1021,-3721){\makebox(0,0)[lb]{\smash{{\SetFigFontNFSS{12}{14.4}{\rmdefault}{\mddefault}{\updefault}{\color[rgb]{0,0,0}$G(a)$}%
}}}}
\put(2461,-3916){\makebox(0,0)[lb]{\smash{{\SetFigFontNFSS{12}{14.4}{\rmdefault}{\mddefault}{\updefault}{\color[rgb]{0,0,0}$\sigma$}%
}}}}
\put(3841,-3226){\makebox(0,0)[lb]{\smash{{\SetFigFontNFSS{12}{14.4}{\rmdefault}{\mddefault}{\updefault}{\color[rgb]{0,0,0}$\sigma G(a)$}%
}}}}
\put(1081,-2161){\makebox(0,0)[lb]{\smash{{\SetFigFontNFSS{12}{14.4}{\rmdefault}{\mddefault}{\updefault}{\color[rgb]{0,0,0}$F(a)$}%
}}}}
\put(2371,-2116){\makebox(0,0)[lb]{\smash{{\SetFigFontNFSS{12}{14.4}{\rmdefault}{\mddefault}{\updefault}{\color[rgb]{0,0,0}$\sigma$}%
}}}}
\put(3766,-2641){\makebox(0,0)[lb]{\smash{{\SetFigFontNFSS{12}{14.4}{\rmdefault}{\mddefault}{\updefault}{\color[rgb]{0,0,0}$F(a)\sigma$}%
}}}}
\put(1636,-1036){\makebox(0,0)[lb]{\smash{{\SetFigFontNFSS{12}{14.4}{\rmdefault}{\mddefault}{\updefault}{\color[rgb]{0,0,0}$H(F(a))$}%
}}}}
\put(4501,-601){\makebox(0,0)[lb]{\smash{{\SetFigFontNFSS{12}{14.4}{\rmdefault}{\mddefault}{\updefault}{\color[rgb]{0,0,0}$\tau$}%
}}}}
\put(3901,-1126){\makebox(0,0)[lb]{\smash{{\SetFigFontNFSS{12}{14.4}{\rmdefault}{\mddefault}{\updefault}{\color[rgb]{0,0,0}$\sigma$}%
}}}}
\put(4276,-2176){\makebox(0,0)[lb]{\smash{{\SetFigFontNFSS{12}{14.4}{\rmdefault}{\mddefault}{\updefault}{\color[rgb]{0,0,0}$H(F(a))$}%
}}}}
\put(3076,-1516){\makebox(0,0)[lb]{\smash{{\SetFigFontNFSS{12}{14.4}{\rmdefault}{\mddefault}{\updefault}{\color[rgb]{0,0,0}$T$}%
}}}}
\put(5131,-1651){\makebox(0,0)[lb]{\smash{{\SetFigFontNFSS{12}{14.4}{\rmdefault}{\mddefault}{\updefault}{\color[rgb]{0,0,0}$H(\sigma)$}%
}}}}
\put(6781,-1471){\makebox(0,0)[lb]{\smash{{\SetFigFontNFSS{12}{14.4}{\rmdefault}{\mddefault}{\updefault}{\color[rgb]{0,0,0}$H(F(a))\tau$}%
}}}}
\put(5491,-2806){\makebox(0,0)[lb]{\smash{{\SetFigFontNFSS{12}{14.4}{\rmdefault}{\mddefault}{\updefault}{\color[rgb]{0,0,0}$H(F(a)\sigma)$}%
}}}}
\put(4516,-5116){\makebox(0,0)[lb]{\smash{{\SetFigFontNFSS{12}{14.4}{\rmdefault}{\mddefault}{\updefault}{\color[rgb]{0,0,0}$\tau$}%
}}}}
\put(3631,-4726){\makebox(0,0)[lb]{\smash{{\SetFigFontNFSS{12}{14.4}{\rmdefault}{\mddefault}{\updefault}{\color[rgb]{0,0,0}$\sigma$}%
}}}}
\put(976,-6721){\makebox(0,0)[lb]{\smash{{\SetFigFontNFSS{12}{14.4}{\rmdefault}{\mddefault}{\updefault}{\color[rgb]{0,0,0}$G(a)$}%
}}}}
\put(1006,-8131){\makebox(0,0)[lb]{\smash{{\SetFigFontNFSS{12}{14.4}{\rmdefault}{\mddefault}{\updefault}{\color[rgb]{0,0,0}$F(a)$}%
}}}}
\put(1651,-9391){\makebox(0,0)[lb]{\smash{{\SetFigFontNFSS{12}{14.4}{\rmdefault}{\mddefault}{\updefault}{\color[rgb]{0,0,0}$H(F(a))$}%
}}}}
\put(4321,-9391){\makebox(0,0)[lb]{\smash{{\SetFigFontNFSS{12}{14.4}{\rmdefault}{\mddefault}{\updefault}{\color[rgb]{0,0,0}$\tau$}%
}}}}
\put(6661,-8896){\makebox(0,0)[lb]{\smash{{\SetFigFontNFSS{12}{14.4}{\rmdefault}{\mddefault}{\updefault}{\color[rgb]{0,0,0}$H(F(a))\tau$}%
}}}}
\put(1621,-4861){\makebox(0,0)[lb]{\smash{{\SetFigFontNFSS{12}{14.4}{\rmdefault}{\mddefault}{\updefault}{\color[rgb]{0,0,0}$H(G(a)$}%
}}}}
\put(1561,-5581){\makebox(0,0)[lb]{\smash{{\SetFigFontNFSS{12}{14.4}{\rmdefault}{\mddefault}{\updefault}{\color[rgb]{0,0,0}$H(G(a)$}%
}}}}
\put(4261,-3841){\makebox(0,0)[lb]{\smash{{\SetFigFontNFSS{12}{14.4}{\rmdefault}{\mddefault}{\updefault}{\color[rgb]{0,0,0}$H(G(a)$}%
}}}}
\put(5281,-4471){\makebox(0,0)[lb]{\smash{{\SetFigFontNFSS{12}{14.4}{\rmdefault}{\mddefault}{\updefault}{\color[rgb]{0,0,0}$H(\sigma)$}%
}}}}
\put(6781,-4291){\makebox(0,0)[lb]{\smash{{\SetFigFontNFSS{12}{14.4}{\rmdefault}{\mddefault}{\updefault}{\color[rgb]{0,0,0}$H(G(a))\tau$}%
}}}}
\put(6796,-6046){\makebox(0,0)[lb]{\smash{{\SetFigFontNFSS{12}{14.4}{\rmdefault}{\mddefault}{\updefault}{\color[rgb]{0,0,0}$H(G(a))\tau$}%
}}}}
\end{picture}%

\end{center}
\caption{Obstructions to deforming the equality between a postcompostion of a natural transformation by a functor and a natural transformation are cocycles \label{post1composition_obstructions_are_cocycles}}
\end{figure}

\begin{figure}
\begin{center}

\setlength{\unitlength}{4144sp}%
\begingroup\makeatletter\ifx\SetFigFontNFSS\undefined%
\gdef\SetFigFontNFSS#1#2#3#4#5{%
  \reset@font\fontsize{#1}{#2pt}%
  \fontfamily{#3}\fontseries{#4}\fontshape{#5}%
  \selectfont}%
\fi\endgroup%
\begin{picture}(4974,5145)(904,-5071)
\thinlines
{\color[rgb]{0,0,0}\put(2131,-4461){\vector( 3,-1){1005}}
}%
{\color[rgb]{0,0,0}\put(3226,-4761){\vector( 1, 0){1215}}
}%
{\color[rgb]{0,0,0}\put(4470,-4750){\vector( 4, 3){1396}}
}%
{\color[rgb]{0,0,0}\put(931,-3681){\vector( 3,-2){1185}}
}%
{\color[rgb]{0,0,0}\put(2131,-591){\vector( 3, 1){1005}}
}%
{\color[rgb]{0,0,0}\put(3226,-291){\vector( 1, 0){1215}}
}%
{\color[rgb]{0,0,0}\put(4470,-302){\vector( 4,-3){1396}}
}%
{\color[rgb]{0,0,0}\put(931,-1371){\vector( 3, 2){1185}}
}%
{\color[rgb]{0,0,0}\put(916,-3656){\vector( 1, 0){1335}}
}%
{\color[rgb]{0,0,0}\put(2326,-3641){\vector( 1, 0){1005}}
}%
{\color[rgb]{0,0,0}\put(3406,-3641){\vector( 1, 0){1080}}
}%
{\color[rgb]{0,0,0}\put(4516,-3641){\vector( 1, 0){1350}}
}%
{\color[rgb]{0,0,0}\put(916,-3641){\vector( 3, 2){1185}}
}%
{\color[rgb]{0,0,0}\put(916,-1396){\vector( 1, 0){1335}}
}%
{\color[rgb]{0,0,0}\put(2326,-1411){\vector( 1, 0){1005}}
}%
{\color[rgb]{0,0,0}\put(3406,-1411){\vector( 1, 0){1080}}
}%
{\color[rgb]{0,0,0}\put(4516,-1411){\vector( 1, 0){1350}}
}%
{\color[rgb]{0,0,0}\put(2116,-2861){\vector( 3, 1){1005}}
}%
{\color[rgb]{0,0,0}\put(4455,-2572){\vector( 4,-3){1396}}
}%
{\color[rgb]{0,0,0}\put(916,-1411){\vector( 3,-2){1185}}
}%
{\color[rgb]{0,0,0}\put(2116,-2191){\vector( 3,-1){1005}}
}%
{\color[rgb]{0,0,0}\put(3211,-2491){\vector( 1, 0){1215}}
}%
{\color[rgb]{0,0,0}\put(4455,-2480){\vector( 4, 3){1396}}
}%
\put(1036,-796){\makebox(0,0)[lb]{\smash{{\SetFigFontNFSS{12}{14.4}{\rmdefault}{\mddefault}{\updefault}{\color[rgb]{0,0,0}$F(a)$}%
}}}}
\put(2116,-226){\makebox(0,0)[lb]{\smash{{\SetFigFontNFSS{12}{14.4}{\rmdefault}{\mddefault}{\updefault}{\color[rgb]{0,0,0}$G(F(a))$}%
}}}}
\put(2266,-2041){\makebox(0,0)[lb]{\smash{{\SetFigFontNFSS{12}{14.4}{\rmdefault}{\mddefault}{\updefault}{\color[rgb]{0,0,0}$H(F(a)$}%
}}}}
\put(3706,-61){\makebox(0,0)[lb]{\smash{{\SetFigFontNFSS{12}{14.4}{\rmdefault}{\mddefault}{\updefault}{\color[rgb]{0,0,0}$\tau$}%
}}}}
\put(3736,-1201){\makebox(0,0)[lb]{\smash{{\SetFigFontNFSS{12}{14.4}{\rmdefault}{\mddefault}{\updefault}{\color[rgb]{0,0,0}$\sigma$}%
}}}}
\put(5266,-436){\makebox(0,0)[lb]{\smash{{\SetFigFontNFSS{12}{14.4}{\rmdefault}{\mddefault}{\updefault}{\color[rgb]{0,0,0}$G(F(a))\tau$}%
}}}}
\put(4591,-1201){\makebox(0,0)[lb]{\smash{{\SetFigFontNFSS{12}{14.4}{\rmdefault}{\mddefault}{\updefault}{\color[rgb]{0,0,0}$G(F(a))\sigma$}%
}}}}
\put(5281,-2056){\makebox(0,0)[lb]{\smash{{\SetFigFontNFSS{12}{14.4}{\rmdefault}{\mddefault}{\updefault}{\color[rgb]{0,0,0}$\sigma H(F(a))$}%
}}}}
\put(1546,-1261){\makebox(0,0)[lb]{\smash{{\SetFigFontNFSS{12}{14.4}{\rmdefault}{\mddefault}{\updefault}{\color[rgb]{0,0,0}$F(a)$}%
}}}}
\put(1066,-2011){\makebox(0,0)[lb]{\smash{{\SetFigFontNFSS{12}{14.4}{\rmdefault}{\mddefault}{\updefault}{\color[rgb]{0,0,0}$F(a)$}%
}}}}
\put(1111,-3091){\makebox(0,0)[lb]{\smash{{\SetFigFontNFSS{12}{14.4}{\rmdefault}{\mddefault}{\updefault}{\color[rgb]{0,0,0}$F(a)$}%
}}}}
\put(1561,-3526){\makebox(0,0)[lb]{\smash{{\SetFigFontNFSS{12}{14.4}{\rmdefault}{\mddefault}{\updefault}{\color[rgb]{0,0,0}$F(a)$}%
}}}}
\put(1171,-4321){\makebox(0,0)[lb]{\smash{{\SetFigFontNFSS{12}{14.4}{\rmdefault}{\mddefault}{\updefault}{\color[rgb]{0,0,0}$F(a)$}%
}}}}
\put(2356,-1231){\makebox(0,0)[lb]{\smash{{\SetFigFontNFSS{12}{14.4}{\rmdefault}{\mddefault}{\updefault}{\color[rgb]{0,0,0}$G(F(a))$}%
}}}}
\put(2311,-3001){\makebox(0,0)[lb]{\smash{{\SetFigFontNFSS{12}{14.4}{\rmdefault}{\mddefault}{\updefault}{\color[rgb]{0,0,0}$H(F(a)$}%
}}}}
\put(2461,-3511){\makebox(0,0)[lb]{\smash{{\SetFigFontNFSS{12}{14.4}{\rmdefault}{\mddefault}{\updefault}{\color[rgb]{0,0,0}$H(F(a)$}%
}}}}
\put(2071,-4921){\makebox(0,0)[lb]{\smash{{\SetFigFontNFSS{12}{14.4}{\rmdefault}{\mddefault}{\updefault}{\color[rgb]{0,0,0}$G(F(a))$}%
}}}}
\put(3796,-3541){\makebox(0,0)[lb]{\smash{{\SetFigFontNFSS{12}{14.4}{\rmdefault}{\mddefault}{\updefault}{\color[rgb]{0,0,0}$\tau$}%
}}}}
\put(3706,-5056){\makebox(0,0)[lb]{\smash{{\SetFigFontNFSS{12}{14.4}{\rmdefault}{\mddefault}{\updefault}{\color[rgb]{0,0,0}$\tau$}%
}}}}
\put(3676,-2386){\makebox(0,0)[lb]{\smash{{\SetFigFontNFSS{12}{14.4}{\rmdefault}{\mddefault}{\updefault}{\color[rgb]{0,0,0}$\sigma$}%
}}}}
\put(5251,-2971){\makebox(0,0)[lb]{\smash{{\SetFigFontNFSS{12}{14.4}{\rmdefault}{\mddefault}{\updefault}{\color[rgb]{0,0,0}$\sigma H(F(a))$}%
}}}}
\put(5056,-4636){\makebox(0,0)[lb]{\smash{{\SetFigFontNFSS{12}{14.4}{\rmdefault}{\mddefault}{\updefault}{\color[rgb]{0,0,0}$G(F(a))\tau$}%
}}}}
\put(4666,-3436){\makebox(0,0)[lb]{\smash{{\SetFigFontNFSS{12}{14.4}{\rmdefault}{\mddefault}{\updefault}{\color[rgb]{0,0,0}$\tau H(F(a))$}%
}}}}
\end{picture}%

\end{center}
\caption{Obstructions to deforming the equality between a precompostion of a natural transformation by a functor and a natural transformation are cocycles \label{pre1composition_obstructions_are_cocycles}}
\end{figure}

\newpage

\section{Partially Trivial Deformations}

In \cite{Yetter} deformations of a functor $F$ (resp. a natural transformation $\sigma:F\Rightarrow G$) in which the source and target are left undeformed (or to say the same thing differently deformed trivially in the strong sense) were shown to be governed by the Hochschild complex $C^\bullet(F)$ (resp. $C^\bullet(F,G)$).


For the desired reduction we will need to consider deformations of pasting diagrams in which a natural transformation (in particular the identity natural transformation) is deformed trivially in the strong sense while its domains and codomains are deformed, possibly nontrivially.

As a warm-up and for potential use in other applications, let us consider first the problem of deforming the (pasting) diagram $D$ consisting of two categories $\cal A$ and $\cal B$ and a functor $F:{\cal A}\rightarrow {\cal B}$, subject to the requirement that $F$ be deformed trivially.

Without the restriction that $F$ be deformed trivially, a(n $n^{th}$ order) deformation would be determined by a family of parallels for the set of operations $\{\mu, \nu, F\}$, the compositions in $\cal A$ and $\cal B$ and the arrow-part of $F$ satisfying the usual cocycle and cobounding conditions.  The restriction that $F$ be deformed trivially requires that all of the positive order parallels to $F$ be zero.

Cohomologically, the triples of degree $k$ parallels $(\mu^{(k)}, \nu^{(k)}, F^{(k)})$ (in the absence of the triviality requirement) lie in ${\frak C}^\bullet(D)$, the mapping cone on $-F_*(p_1) + F^*(p_2): C^\bullet({\cal A}) \oplus  C^\bullet({\cal B}) \rightarrow C^\bullet(F)$ and satisfy cocycle ($k = 1$) and cobounding ($k > 1$) conditions.  The requirement that $F^{(1)} = 0$ together with the coboundary condition thus implies that $(\mu^{(1)}, \nu^{(1)})$ lies in 
$ker(-F_*(p_1) + F^*(p_2))$.  

In fact we have

\begin{proposition}
First order deformations of the (pasting) diagram $D$ consisting of two categories $\cal A$ and $\cal B$ and a functor $F:{\cal A}\rightarrow {\cal B}$, subject to the requirement that $F$ be deformed trivially are classified by the second cohomology of $ker(-F_*(p_1) + F^*(p_2))$.  Moreover all obstructions to deforming $D$ with $F$ deformed trivially are cocycles in the
third cochain group of $ker(-F_*(p_1) + F^*(p_2))$ and an $n^{th}$ order deformation with $F$ deformed trivially can be extended to an $(n+1)^{st}$ order deformation with $F$ deformed trivially if and only if there is a degree 2 cochain in $ker(-F_*(p_1) + F^*(p_2))$ cobounding the degree $n$ obstruction.
\end{proposition}

\begin{proof}
The classification of first order deformations follows from the remarks above.  The rest of the proposition follows from the corresponding result in \cite{Yetter} without the triviality restriction once it is shown that all obstructions, which {\em a priori} lie in the cone on $-F_*(p_1) + F^*(p_2)$ lie in the kernel (considered as a subcomplex under the obvious inclusion).

As observed above, $(\mu^{(1)}, \nu^{(1)})$ lies in 
$ker(-F_*(p_1) + F^*(p_2))$.  So suppose as an induction hypothesis that for all $k < n$
$(\mu^{(k)}, \nu^{(k)})$ lies in 
$ker(-F_*(p_1) + F^*(p_2))$ (or equivalently and more usefully $F(\mu^{(k)}) = \nu^{(k)}(F,F)$).

The degree $n$ obstruction $\omega^{(n)}$ in the cone then has coordinates

\[
 \sum_{\stackrel{ k,\ell < n}{ k + \ell = n }}
[ \mu^{(k)}(\mu^{(\ell)}(-,-),-) - \mu^{(k)}(-,\mu^{(\ell)}(-,-)) ], \]

\[ \sum_{\stackrel{ k,\ell < n}{ k + \ell = n }}
[ \nu^{(k)}(\nu^{(\ell)}(-,-),-) - \nu^{(k)}(-,\nu^{(\ell)}(-,-)) ], \]

\noindent and
\[ \sum_{\stackrel{ k,\ell < n}{ k + \ell = n }} F^{(k)}(\mu^{(\ell)}) -
\sum_{\stackrel{ k,\ell,m < n}{ k + \ell + m = n }} \nu^{(k)}(F^{(\ell)},F^{(m)})  \]

The last vanishes since each term involves $F^{(k)} = 0$ for some $k > 0$ (and in the latter
sum all the $\nu^{(k)}$ are bilinear).  It thus remains only to show that $F_*$ of the first coordinate equals $F^*$ of the second.

Computing

\[ F( \sum_{\stackrel{ k,\ell < n}{ k + \ell = n }}
[ \mu^{(k)}(\mu^{(\ell)}(-,-),-) - \mu^{(k)}(-,\mu^{(\ell)}(-,-)) ]) \] 

\[ =  
\sum_{\stackrel{ k,\ell < n}{ k + \ell = n }} 
[ \nu^{(k}(F(\mu^{(\ell)}(-,-),F(-)) - \nu^{(k}(F(-),F(\mu^{(\ell)}(-,-))) ] \]

\[ =  \sum_{\stackrel{ k,\ell < n}{ k + \ell = n }} 
[ \nu^{(k)}(\nu^{(\ell)}(F(-),F(-)),F(-)) - \nu^{(k)}(F(-),\nu^{(\ell)}(F(-),F(-))) ] \]

In each case the equality holds by the induction hypothesis, and in the second case
the bilinearity $\nu^{(k)}$ for each $k$.  Thus the proposition holds.
\end{proof}

A similar result holds for deformations of a diagram consisting of two parallel functors $F,G:{\cal A}\rightarrow {\cal B}$ and a natural transformation $\sigma:F\Rightarrow G$ between them which is to be deformed trivially:  without the triviality restriction, the deformation complex is the cone on $\sigma \ddag$.  With the restriction, $ker(\sigma \ddag)$ becomes the deformation complex:

\begin{proposition}First order deformations of the pasting diagram $D$ consisting of two categories a pair of parallel functors $F,G:{\cal A}\rightarrow {\cal B}$, their sources and target and a natural transformation $\sigma:F\Rightarrow G$ subject to the requirement that $\sigma$ be deformed trivially are classified by the second cohomology of $ker(\sigma \ddag)$.  Moreover all obstructions to deforming $D$ with $\sigma$ deformed trivially are cocycles in the
third cochain group of $ker(\sigma \ddag)$ and an $n^{th}$ order deformation with $\sigma$ deformed trivially can be extended to an $(n+1)^{st}$ order deformation with $\sigma$ deformed trivially if and only if there is a degree 2 cochain in $ker(\sigma \ddag)$ cobounding the degree $n$ obstruction.
\end{proposition}

\begin{proof}
The argument is essentially identical to that for the previous proposition, except for the calculation showing that obstructions, which {\em a priori} lie in the full deformation complex of $D$ lie in $ker(\sigma \ddag)$.

To see that all obstructions lie in $ker(\sigma \ddag)$, consider Figure \ref{naturality_obstructions_are_cocycles}.  The condition that degree $n$ obstruction lies in $ker(\sigma \ddag)$ is precisely the condition that the signed sum of the degree $n$ obstruction-type expressions associated to the non-trivial faces other than the three hexagons vanish, when all higher order parallels of $\sigma$ are instantiated as 0.  But observe that the other faces are all trivial or evaluate to 0,  in all degrees greater than 0, trivially under the hypothesis that all higher order parallels of $\sigma$ are 0, so the result follows from the general principles of the polygonal method.
\end{proof}

Applying the same argument {\em mutatis mutandis} to Figure \ref{nat_from_composition_obstructions_are_cocycles} shows

\begin{proposition}\label{trivialtriangle} 
First order deformations of the pasting diagram $D$ consisting of three categories  ${\cal A}, {\cal B}$ and ${\cal C}$ three functors $F:{\cal A}\rightarrow {\cal B}$, $G:{\cal B}\rightarrow {\cal C}$ and $H:{\cal A}\rightarrow {\cal C}$ and a natural transformation $\sigma:G(F)\Rightarrow H$ subject to the requirement that $\sigma$ be deformed trivially are classified by the second cohomology of $ker(\phi)$, where

\[ \phi = [0, 0, (-)\{\sigma\}, G^*(\sigma_*), F_*(\sigma_*), -\sigma^*]:{\frak C}^\bullet(\partial D) \rightarrow C^\bullet(G(F),H) .\]  

\noindent where $\partial D$ is the diagram obtained by omitting $\sigma$.

Moreover all obstructions to deforming $D$ with $\sigma$ deformed trivially are cocycles in the
third cochain group of $ker(\phi)$ and an $n^{th}$ order deformation with $\sigma$ deformed trivially can be extended to an $(n+1)^{st}$ order deformation with $\sigma$ deformed trivially if and only if there is a degree 2 cochain in $ker(\phi)$ cobounding the degree $n$ obstruction.
\end{proposition}

In the special case where $H = G(F)$ and $\sigma = Id_{G(F)}$, it is intuitively clear that the deformations of $D$ in which $\sigma$ is deformed trivially are completely determined by
the deformation complex of the simpler diagram 
${\cal A} \stackrel{F}{\rightarrow} {\cal B} \stackrel{G}{\rightarrow} {\cal C}$.  In fact we have

\begin{proposition} \label{reducestotwoedges}
In the deformation complex of the pasting diagram $D$ consisting of three categories  ${\cal A}, {\cal B}$ and ${\cal C}$ three functors $F:{\cal A}\rightarrow {\cal B}$, $G:{\cal B}\rightarrow {\cal C}$ and $H=G(F):{\cal A}\rightarrow {\cal C}$ and the natural transformation $\sigma = Id_{G(F)}:G(F)\Rightarrow H$ the complex $ker(\phi)$ where $\phi$ is as in Proposition \ref{trivialtriangle} is isomorphic to the deformation complex of the subdiagram
${\cal A} \stackrel{F}{\rightarrow} {\cal B} \stackrel{G}{\rightarrow} {\cal C}$.
\end{proposition}

\begin{proof}
The obvious quotient map from the deformation complex of $\partial D$ to 

\[ C^\bullet(G(F),G(F)) = C^{\bullet}(G(F)) \]

\noindent is split by the obvious inclusion.  The kernel of the quotient map is easily seen to be isomorphic to $ker(\phi)$, while the quotient of the inclusion is easily seen to be the deformation complex of ${\cal A} \stackrel{F}{\rightarrow} {\cal B}
\stackrel{G}{\rightarrow} {\cal C}$.
\end{proof}

Combining Propositions \ref{uniontopushout}, \ref{trivialtriangle} and \ref{reducestotwoedges} then gives

\begin{proposition} \label{trivialthreecell}

Let $D$ be a pasting diagram which could arise by the following construction:  begin with a computad consisting of a composable 1-pasting diagram with any number of edges, a single
edge with the same domain and codomain as the 1-pasting diagram, a pair of 2-cells both with the resulting circle as boundary (and a 3-cell, or not).  Now, triangulate each 2-cell in a way corresponding to any parenthesization of the edges of the original 1-pasting diagram.  Apply a map of pasting schemes to the underlying pasting scheme of $k-lincat$ such that the value of every 2-cell is an identity natural transformation with domain given by two composable functors and codomain given by their composition.

If $D$ is such a pasting diagram, then the single edge in the initial part of the construction is labeled with the composition of the initial 1-pasting diagram, and if, moreover, $\Phi$ is the chain map from the deformations complex ${\frak C}^\bullet(D)$ to the direct sum indexed by the triangular 2-cells of $D$ whose coordinate for a 2-cell $K$ with edges labeled by a composable
pair of functors $F:{\cal A}\rightarrow {\cal B}$ and $G:{\cal B}\rightarrow {\cal C}$ and their composition $G(F)$ is given by projection onto ${\frak C}^\bullet(\partial K)$ followed by the map $\phi$ of Proposition \ref{trivialtriangle} onto $C^\bullet(G(F),G(F)$, then $ker(\Phi)$ classifies the deformations of $D$ in which all of the identity natural transformations are deformed trivially, and is isomorphic to the deformation complex of the original composable 1-pasting diagram.

\end{proposition}

\begin{proof}
The key thing to note here is that because all triangulations of disks are shellable, the pasting diagram can be constructed out of the trivial triangles of Proposition \ref{trivialtriangle} by iterated pushouts with diagrams previously so constructed until a pushout of the two triangulated disks along their boundary is made, and that once the triviality condition are imposed the presence or absence of the 3-cell makes no difference to the kernel.
\end{proof}

Note also that a 2-dual form of Proposition \ref{trivialthreecell} in which all of the identity natural transformations have the composition as their domain and two composable functors as their codomain also holds with the same proof.


The proofs of Propositions \ref{trivialtriangle} and \ref{reducestotwoedges} can be extended to show

\begin{proposition} \label{fatpost1comp}
Let $D$ be the pasting diagram consisting of three categories ${\cal A}, {\cal B}, {\cal C}$, functors $F,G:{\cal A}\rightarrow {\cal B}, H:{\cal B}\rightarrow {\cal C}$, the composite functors $H(F)$ and $H(G)$, natural transformations $\sigma:F\Rightarrow G$, $\tau:H(F)\Rightarrow H(G)$ labeling bigons, and $Id_{H(F)}$ and $Id_{H(G)}$ labeling triangles with two edges labeled by the composands as domain and one edge labeled by the composition as codomain, and a single 3-cell asserting the commutativity of the diagram.  Letting $\Phi$ be the map with one coordinate for each triangle given by $\phi$ of Proposition \ref{trivialtriangle} for that 2-cell, $ker(\Phi)$ classifies the deformations of $D$ in which the two identity natural transformations are deformed trivially, and is isomorphic to the complex of Example \ref{equalspost1comp}.
\end{proposition}

\noindent and

\begin{proposition} \label{fatpre1comp}
Let $D$ be the pasting diagram consisting of three categories ${\cal A}, {\cal B}, {\cal C}$, functors $F:{\cal A}\rightarrow {\cal B}, G,H:{\cal B}\rightarrow {\cal C}$, the composite functors $F(G)$ and $F(H)$, natural transformations $\sigma:F\Rightarrow G$, $\tau:F(G)\Rightarrow F(H)$ labeling bigons, and $Id_{F(G)}$ and $Id_{F(H)}$ labeling triangles with two edges labeled by the composands as domain and one edge labeled by the composition as codomain, and a single 3-cell asserting the commutativity of the diagram.  Letting $\Phi$ be the map with one coordinate for each triangle given by $\phi$ of Proposition \ref{trivialtriangle} for that 2-cell, $ker(\Phi)$ classifies the deformations of $D$ in which the two identity natural transformations are deformed trivially, and is isomorphic to the complex of Example \ref{equalspre1comp}.
\end{proposition}

Observe, that the conclusion of both Proposition \ref{fatpost1comp} and \ref{fatpre1comp} hold regardless of which part of the boundary (two composable functors or their composition) is the domain and which is the codomain of each of the triangular faces labeled with identity natural transformations.  

\section{All Obstructions are Cocycles}

We are now in a position to prove the main result:  that the deformations of any pasting diagram are classified by deformation complex in which all obstructions to extending a deformation of order $n$ to a deformation of order $n+1$ are cocycles.  To do this we will replace the diagram with a more complex diagram with simpler parts, whose deformation theory, when some of the 2-cells are required to be deformed trivially, coincides with that of the original pasting diagram.

\begin{definition}
A pasting diagram is {\em finely divided} if every 2-cell is either a bigon or a triangle labeled with the identity natural transformation, and every 3-cell is of one of the following forms:

\begin{itemize}
\item a composition-free pillow with bigonal cross-section (imposing the equality of two natural transformations between the same pair of functors)
\item a diagram of the form in Proposition \ref{trivialthreecell} or its 2-dual
\item a "triangular pillow" as in Example \ref{triangularpillow}
\item a diagram of the form in Example \ref{equalspost1comp}
\item a diagram of the form in Proposition \ref{fatpost1comp}
\item a diagram of the form in Example \ref{equalspre1comp}
\item a diagram of the form in Proposition \ref{fatpre1comp}
\end{itemize}
\end{definition}

\begin{proposition}
Every pasting diagram $D$ can be replaced with a finely divided pasting diagram $f(D)$ whose underlying cell complex is a subdivision of the underlying cell complex of $D$ and whose deformations when all identity natural transformations labeling triangles are deformed trivially are equivalent to deformations of $D$.
\end{proposition}

\begin{proof}  Begin by subdividing each 2-cell as follows:  insert a bigon with 1-cells labeled by the 1-composition of the domain and codomain of the 2-cell and labeled by the same natural transformation as the original 2-cell, making the domain (resp. codomain) of the bigon coincide with that of the original 2-cell if the domain (resp. codomain) consists of a single 1-cell.  Label the complementary cell(s) with identity natural transformations to be deformed trivially.  

Now each of the complementary cells has a domain (resp. codomain) which is a sequence of 1-cells labeled by functors and a codomain (resp. domain) which is a single 1-cell labeled with their composition.   Choose a parenthesization of the composition and subdivide the 2-cell in the usual way corresponding to a parenthesization, labeling the new 1-cells with the appropriate pairwise compositions of labels already present and all 2-cells with identity natural transformation (to be deformed trivially).

Now, for each 3-cell which is not already of one of the forms specified in the proposition, choose an order of pasting composition for the domain (resp. codomain), and use this to create a subdivision of the 3-cell as follows:

Iteratively, for each 1-composition of a natural transformation with a functor, insert a diagram of the form in Proposition \ref{fatpost1comp} or \ref{fatpre1comp} as appropriate and for each binary 2-composition, insert a triangular pillow.  When this has been done for both the domain and codomain, there will be bigons labeled with the pasting compostion of the domain and codomain of the original 3-cell.  Identify their domains (resp. codomains).  The 3-cells of the resulting subdivision are now those explicitly inserted in the construction, a composition-free pillow with bigonal cross-section, and, we claim, 3-cells all of the form in Proposition \ref{trivialthreecell}.

Verifying the claim is a matter of keeping track of the triangular faces in such a way as to organize them into
the pairs of triangulated disks with appropriate boundaries:  Note that immediately after the subdivision of 2-cells into a bigon and triangles, the triangulated cell(s) (if any -- the original 2-cell could have been a bigon) are of the desired form:  they have a domain (resp. codomain) consisting of a composable 1-pasting diagram and codomain (resp. domain) consisting of a single edge labeled with the composition. Call such a triangulated 2-cell a ``nicely triangulated 2-cell''.  Thus the remainder of the construction begins with the original 2-cells decomposed into a family of nicely triangulated 2-cells and family of bigons.  Throughout the rest of the construction the family of nicely triangulated 2-cells will be iteratively replaced with families with fewer nicely triangulated 2-cells (but with more triangles in the newer ones).  For addition of a 3-cell of the forms in Propositions \ref{fatpost1comp} or \ref{fatpre1comp}, the triangles will have as two of their edges the single edge labeled with the composition from previously constructed nicely triangulated 2-cells.  For each of the triangles, replace these two nicely triangulated 2-cells with the union of the triangle and the two nicely triangulated 2-cells, noticing that it is nicely triangulated.
When the triangular pillows and the bigonal pillow are adjoined, the family of nicely triangulated 2-cells is unchanged.

Thus, when the construction ends, the remaining 3-cells in the decomposition must each be bounded by the union of two nicely triangulated 2-cells.
\end{proof}

We can now prove the main result (recalling the dimension conventions of \cite{Yetter} which leaves the deformation cohomology of a pasting digram in dimension $-1$) :

\begin{theorem} \label{main}
For any pasting diagram $D$ of $k$-linear categories, $k$-linear functors and natural transformations, the deformations of $D$ are classified by the $-1$-cohomology of the $ker(\Phi)$, where $\Phi$ is the map from the deformation complex ${\frak C}^\bullet(f(D))$ given by the same description as in Proposition \ref{trivialthreecell}, and all obstructions to extension of a degree $n$ deformation to a degree $n+1$ deformation are 0-cocycles.
\end{theorem}

\begin{proof}
For the classification statement, first observe that the proof of Proposition \ref{uniontopushout} applies equally well to taking unions of pasting diagrams in which some 2-cells have been specified as deforming trivially (provided that cells in the intersection are specified as deforming trivially in both pasting diagrams of which the union is being taken).  The isomorphisms of Propositions \ref{trivialthreecell}, \ref{fatpost1comp} and \ref{fatpre1comp} then combine by universal property of pushouts to give an isomorphism between ${\frak C}^\bullet(D)$ and $ker(\Phi)$ in the category of chain complexes.  

That all obstructions are cocycles is a condition local to each cell together with its boundary, and all cells of $f(D)$ are either deformed trivially, so that the obstruction vanishes by triviality, are of the forms which were shown to have vanishing obstructions in Theorem \ref{singlecomposition}, or are composition-free and thus have vanishing obstructions by the results of \cite{Yetter}.
\end{proof}

\section{Concluding Remarks}

The primary purpose of this paper, to complete the deformation theory for pasting diagrams of $k$-linear categories described in \cite{Yetter}, was accomplished by Theorem \ref{main}.  Its primary importance, however, may lie less in the result than in the techniques used.  It is the authors' intent to apply the deformation theory of pasting diagrams to the still-open problem of providing a complete deformation theory for monoidal categories in which all arrow-valued elements are deformed simultaneously, rather than just the structure maps as in \cite{Yetter_book}.  This will require using extensions of ${\mathbb T}(D)$ hinted at above, in which monoidal prolongations are included as additional operations, as the basis of the polygonal method.  It will also avoid the difficulties in \cite{Shrestha} arising from the need to intuit the correct formulas for higher differentials in a multicomplex from the scant data provided by the instances arising in the deformation theory.

Likewise, although for the present purpose there is something unsatisfying about our detour though finely divided pasting diagrams with specified trivially-deformed identity cells -- we believe the direct result, that all obstructions in the deformation complex of any pasting diagram as defined in \cite{Yetter} are cocycles, though it is beyond our present ability to prove -- like the polygonal technique, this detour may have other applications.  The difficulty with the direct result was that although 2-categories, and the pieces of them corresponding to direct summands in deformation complexes are inherently ``globular'', pasting diagrams are ``opetopic''.  The reduction of the opetopic to trivial (mostly simplicial) elements and globular elements might well find application in connecting the zoo of opetopic definitions of weak n-categories with the globular approach of Batanin.


\begin{thebibliography}{1}

\bibitem{FS} P.\ Freyd and A. Scedrov {\em Categories, Allegories} (North-Holland Mathematical Library volume 39) North-Holland, Amsterdam 1990.

\bibitem{G} M.\ Gerstenhaber, ``On the deformation of rings and algebras'', {\em Annals of Math.} 79 (1) (1964) 59-103.

\bibitem{GS} M.\ Gerstenhaber and S.D.\ Schack, ``On the deformation of algebra morphisms and diagrams,'' {\em Trans. Amer. Math. Soc.}, 279 (1983) 1-50.

\bibitem{GV} M.\ Gerstenhaber and A.A.\ Voronov,
"Homotopy G-algebras and moduli space operad",
{\em Int. Math. Res. Notices} 3 (1995)
141-153.

\bibitem{Power} J.\ Power, ``An $n$-categorical pasting theorem'' in {\em Category Theory} SLNM no. 1488 (A.\ Carboni, et al.\ eds.) (1991) 326-358. 

\bibitem{Shrestha} T.\ Shrestha, {\em Algebraic Deformation of a $k$-linear Monoidal Category}, Kansas State University doctoral dissertation (2010).

\bibitem{Street} R.\ Street, ``The algebra of oriented simplexes'', {\em J. Pure and App. Alg.} 49 (1987) 283-335.

\bibitem{Yetter} D.N.\ Yetter ``On Deformations of Pasting Diagrams'' {\em Theory and Application of Categories} 22 (2009) 24-53.

\bibitem{Yetter_book} D.N.\ Yetter {\em Functorial Knot Theory}, World Scientific (2001).

\end{thebibliography}
\end{document}